\documentclass{amsart}
\usepackage{abstract}
\usepackage{amssymb}
\usepackage{microtype}
\usepackage{tikz}
\usepackage{enumerate}
\usepackage[all]{xy}
\usepackage{theoremref}
\usepackage{comment}
\usepackage{float}
\usepackage[showframe=false, margin = 1.5 in]{geometry}
\usepackage{color}
\usepackage{mathrsfs}
\usepackage{appendix}

\DeclareMathOperator{\Poi}{Poi}

\newcommand{\E}{\mathbb{E}}
\renewcommand{\P}{\mathbb{P}}
\newcommand{\Z}{\mathbb{Z}}

\newcommand{\ep}{\epsilon}

\newcommand{\abs}[2][]{#1\lvert #2 #1\rvert}
\theoremstyle{plain}
\newtheorem{thm}{Theorem}
\newtheorem{lemma}[thm]{Lemma}
\newtheorem{prop}[thm]{Proposition}
\newtheorem{cor}[thm]{Corollary}

\title{Generalized Random Simplicial Complexes}
\author{Christopher F. Fowler}

\begin{document}
\maketitle

\begin{abstract} We consider a multi-parameter model for randomly constructing simplicial complexes that interpolates between random clique complexes and Linial-Meshulam random $k$-dimensional complexes. Unlike these models, multi-parameter complexes exhibit nontrivial homology in numerous dimensions simultaneously. We establish upper and lower thresholds for the appearance of nontrivial cohomology in each dimension and characterize the behavior at criticality.

\end{abstract}

\begin{section}{Introduction}
\begin{subsection}{Background}
Many problems in physics, economics, biology, and mechanics involve the modeling of extremely large and intricate systems. With such high levels of complexity, understanding these systems from their microscopic structure is often intractable. In such cases it may make more sense to view them as random topological spaces with certain probability parameters. This framework enables us to make a variety of powerful conclusions about how these systems will generally behave. Indeed, as mentioned in \cite{survey}, the study of random geometric and topological spaces has on several occasions lent intuition to the extraordinary prevalence of certain properties amongst mathematical objects.

The purpose of this work is to understand the homological behavior of a generalized model for random simplicial complexes, mentioned in \cite{survey} and recently explored in \cite{douchers}. We define $X(n, p_1, p_2, \ldots)$ to be the probability distribution over simplicial complexes on vertex set $[n] = \{1, \ldots, n\}$ whose distribution on $1$-skeletons agrees with $G(n, p_1)$. The distribution on higher dimensional skeletons is constructed inductively: for an integer $k > 1$, any $k$-simplex whose boundary is contained in our complex is added with probability $p_k$. This provides a measure on virtually all simplicial complexes on $n$ vertices. Two well studied structures, Linial-Meshulam and clique complexes, are realized as $X(n, 1, \ldots, 1, p_k, 0, \ldots)$ and $X(n, p, 1, \ldots)$.

The study of random topological spaces began with random graphs, the seminal example of which is $G(n,p)$, the Erd\H{o}s-R\'enyi model. Given a probability parameter $p \in (0,1)$, typically a function of $n$, we consider a graph $G$ on $n$ vertices where every edge between two vertices of $G$ is added independently with probability $p$. This defines a probability measure on the set of all simple graphs on $n$ vertices and we say $G \sim G(n,p)$ to indicate $G$ is a random graph with law $G(n,p)$.

Most random topology results pertain to the asymptotic behavior of a model, ie.\ what happens as the number of vertices tends to infinity. Given some property $\mathcal{A}$ of simplicial complexes, we say that $X \in \mathcal{A}$ \emph{with high probability}, or \emph{w.h.p.},\ if 
\begin{equation*}
\lim_{n\rightarrow \infty} \P \left[X \in \mathcal{A} \right] = 1.
\end{equation*}
A formative result of random graph theory, proven by Erd\H{o}s and R\'enyi in \cite{erdos}, was the sharp threshold of $p = \log n / n$ for connectivity in $G(n,p)$: for any $\omega(n) \rightarrow \infty$ as $n \rightarrow \infty$, if $p \geq (\log n +\omega(n)) / n$ then $G(n,p)$ is w.h.p.\ connected, and if $p \leq (\log n- \omega(n)) / n$ then $G(n,p)$ is w.h.p.\ disconnected.

Significant work has been done on the behavior of random graphs since \cite{erdos}. Providing a higher dimension analog, recent study has been focused on several models for random simplicial complexes. One of the most natural questions to ask, results in this field often depict the homological or cohomological behavior of a complex. Even the connectivity threshold for $G(n,p)$ is a statement about the $0$-homology of graphs: $H_0(G, \Z) = \Z^m$ where $m$ is the number of connected components of $G$.

A high-dimensional analog to $G(n,p)$ is $Y_k(n, p)$, the Linial-Meshulam model for random $k$-dimensional simplicial complexes. We begin with a complex on $n$ vertices and full $(k-1)$-skeleton, then add every possible $k$-face independently with probability $p$. Linial and Meshulam initially considered when $k=2$ in \cite{rand2}, establishing a sharp threshold for when $\Z_2$-homology disappears in the first dimension. Babson, Hoffman, and Kahle later looked at the fundamental group of this model in \cite{fundamental2}, proving a threshold where $\pi_1\left(Y_2(n,p)\right)$ transitions w.h.p.\ from hyperbolic to trivial.

Meshulam and Wallach in \cite{randk} extended the result in \cite{rand2} to $H_{k-1} \left(Y_k(n,p), \Z_q \right)$ for any dimension $k$. Their work was followed by \cite{randk2}, where Hoffman, Kahle, and Paquette demonstrated an upper bound for the vanishing of integer homology in this model. It is also natural to ask how $H_k(Y , \Z)$ behaves in these complexes. Kozlov proved a sharp threshold for the appearance of $k$-homology in \cite{kozlovthreshold}. Aronshtam and Linial in \cite{tophomvanish}, joined by {\L}uczak and Meshulam in \cite{tophomcollapse}, shed further light on the topological structure of these complexes.

Another model of interest is the random clique complex model, $X(n,p)$. Just as in our own model, the distribution of the $1$-skeleton is identical to $G(n,p)$, but in this case the edges dictate the entire complex. Given some $X \sim X(n,p)$, $X$ contains the $k$-simplex spanned by a set of $k+1$ vertices only if the vertices form a complete subgraph in $X$, called a \emph{$(k+1)$-clique}. For any dimension $k$, Kahle established in \cite{general} and \cite{sharp} sharp thresholds for $p$ for which there will be nontrivial $k$-th cohomology, showing that primarily cohomology will w.h.p.\ be nontrivial in just one dimension. Kahle has proved numerous results concerning the behavior of $X(n,p)$, such as establishing a central limit theorem for the distribution of Betti numbers $\beta^k = \textnormal{dim}\left(H^k(X,\mathbb{Q})\right)$ with Meckes in \cite{clt}.

As we noted before, all these complexes are special cases of $X(n, p_1, p_2, \ldots)$. The random graph model $G(n,p)$ is identical to $X(n, p, 0, \ldots)$, $Y_k(n,p)$ corresponds to $X(n, 1, \ldots, 1, p_k=p, 0, \ldots)$, and clique complexes are the case $X(n, p_1, 1, \ldots)$. In fact, many of our results are achieved through a reworking of frameworks laid down in \cite{general} and \cite{sharp}. This appears to be the natural bridge between these models, and we show that often the results for specific models may be extended to this broader construction. Through this process we exhibit cohomological behavior unique to this model.

\end{subsection}

\begin{subsection}{Statement of Results}
Our theorems deal with the $(k-1)$-th homology or cohomology of $X(n, p_1, p_2, \ldots)$. Since the $(k-1)$-th (co)homology of a simplicial complex depends only on its $k$-skeleton, these theorems only depend on probabilities $p_1$ through $p_k$. Repeated application of our theorems for each dimension will often fully describe the cohomology of our random complex. The primary open problem from this work concerns the $(k-1)$-th homology of our complexes when $p_k = 1$, which we discuss following our statement of results.

We present a low-dimension example to give some intuition for where the inequalities in our theorems come from, as well as illustrate the potential for non-trivial cohomology in multiple dimensions simultaneously.
\begin{prop}
\label{1stlemma}
Let $X \sim X(n,p_1, p_2, \ldots)$ with $p_1, p_2, p_3 = n^{-3/8},$ then w.h.p.\ $H^1(X, \mathbb{Q}) \neq 0$ and $H^2(X, \mathbb{Q}) \neq 0$.
\end{prop}
\begin{proof}
We outline a proof of the stronger statement that if $$p_2,p_3 \neq 1,\ 6\alpha_1 + 4 \alpha_2 < 4, \textnormal{ and } 1 \leq 2\alpha_1 + \alpha_2$$ then w.h.p.\ $H^1(X, \mathbb{Q}) \neq 0$ and $H^2(X, \mathbb{Q}) \neq 0$.

Within this proof, and later in Section $5$, we consider the appearance of certain subcomplexes in $X$. First, we establish the presence of triangles with an unfilled $2$-face whose first edge, determined lexicographically, is not part of any $2$-face in $X$. Our complex is defined on the vertex set $[n]$, and for any $j \in \binom{[n]}{3}$ we let $A_j$ denote the event that the vertex set corresponding to $j$ forms such a subcomplex. Using independence, this has probability
\begin{equation*}
\P \left[ A_j \right] = p_1^3 \left(1-p_2 \right) \left(1-p_1^2p_2\right)^{n-3}.
\end{equation*}
The first two terms require the three edges are in $X$ while the $2$-simplex itself is not present. The last term ensures our first edge does not form a $2$-simplex with any of the $n-2$ remaining vertices. 

Letting $M_1$ denote the number of such subcomplexes in $X$, by linearity of expectation
\begin{equation*}
\E [ M_1] = \sum_{j \in \binom{[n]}{3}} \P[ A_j] = \binom{n}{3}p_1^3\left(1-p_2 \right) \left(1-p_1^2p_2\right)^{n-3}.
\end{equation*}
Using standard first moment techniques we see, for large enough $n$,
\begin{align*}
\E [ M_1 ] &\approx \frac{n^3}{6} n^{-3\alpha_1} \left( 1 - p_2 \right) \left( 1 - n^{-(2\alpha_1 + \alpha_2)} \right)^{n} \\
&\approx \frac{1}{6} n^{3-3\alpha_1} \left( 1 - p_2 \right) e^{-n^{1 - (2\alpha_1 + \alpha_2)}}.
\end{align*}
The last  two terms are $\Theta (1)$ when $p_2 \neq 1$ and $1 \leq 2\alpha_1 + \alpha_2$, then $\alpha_1 < 1$ implies $\E[M_1] \rightarrow \infty$. Second moment arguments, detailed in Appendix A, then show that w.h.p.\ $M_1 > 0$.

We now show the existence of tetrahedrons with unfilled $3$-face and first triangle not contained in any $3$-face. For each $l \in \binom{[n]}{3}$, let $B_l$ be the event that the vertices $l$ form such a subcomplex in $X$. Similar considerations show
\begin{equation*}
\P \left[ B_l \right] = p_1^6 p_2^4 \left(1-p_3 \right) \left(1-p_1^3 p_2^3 p_3 \right)^{n-4}.
\end{equation*}
Letting $M_2$ denote the total number of such subcomplexes in $X$, linearity of expectation shows
\begin{equation*}
\E\left[ M_2 \right] = \sum_{l \in \binom{[n]}{4}} \P[B_l] = \binom{n}{4} p_1^6 p_2^4 \left(1-p_3\right) \left(1-p_1^3 p_2^3 p_3 \right)^{n-4}.
\end{equation*}
It follows that if $p_3 \neq 1$, $6 \alpha_1 + 4 \alpha_2 < 4$, and $1 \leq 3\alpha_1 + 2 \alpha_2 +\alpha_3$, then $\E[M_2] \rightarrow \infty$. Second moment calculations establish that w.h.p.\ $M_2 > 0$.

Combining the two sets of requirements on $p_i$ yields that whenever $p_2, p_3 \neq 1$, $1\leq 2\alpha_1 +\alpha_2$, and $6 \alpha_1 + 4 \alpha_2 < 4$ w.h.p.\ $M_1, M_2>0$. Each such subcomplex can be seen to generate a non-trivial $\Z$-summand in the $1$ or $2$-homology, respectively. Thus w.h.p.\ $H_1(X, \Z) \neq 0$ and $H_2(X,\Z) \neq 0$, and our result follows by the Universal Coefficients Theorem, covered in the next section.
\end{proof}

As with clique complexes, the $(k-1)$-cohomology of $X(n, p_1, p_2, \ldots)$ has two phase transitions. We begin with no $(k-1)$-simplices and trivial cohomology, as our probabilities increase cohomology appears, then eventually the $p_i$ become too large and it disappears. Our work revolves around describing the thresholds for these transitions.

The following result establishes when the probabilities are sufficiently large that we will have trivial cohomology.

\begin{thm}
\label{upper}
Let $X \sim X(n, p_1, p_2, \ldots)$ with $p_i = n^{-\alpha_i}$ and $\alpha_i \geq 0$ for all $i$. If
\begin{equation} \label{fofree}
\sum_{i=1}^{k} \alpha_i \binom{k}{i} < 1
\end{equation}
then w.h.p.\ $H^{k-1} ( X, \mathbb{Q}) = 0$.
\end{thm}

We prove this threshold is sharp by showing in the next theorem that on the other side of \eqref{fofree} cohomology is nontrivial. Moreover, the second regime for which cohomology exists is crucial to results such as Proposition \ref{1stlemma}, where $H^k(X,\mathbb{Q}) \neq 0$ for several $k$.

\begin{thm}
\label{middle}
Let $X \sim X(n, p_1, p_2, \ldots)$ with $p_i = n^{-\alpha_i}$, $\alpha_i \geq 0$ for all $i$, and
\begin{equation} \label{midrange1}
1 \leq \sum_{i=1}^{k} \alpha_i \binom{k}{i}.
\end{equation}
If
\begin{equation} \label{midrange2}
\sum_{i=1}^{k-1} \alpha_i \binom{k-1}{i} < 1
\end{equation}
then w.h.p.\ $H^{k-1} ( X, \mathbb{Q} ) \neq 0$. Moreover, if $p_k \neq 1$ we can relax this bound to
\begin{equation} \label{midrange3}
 \sum_1^{k-1} \alpha_i \binom{k+1}{i+1} < k+1.
\end{equation}
\end{thm}

A common question to ask concerning phase transitions is what happens at the boundary between phases. Allowing the $p_i$ to be more varied functions of $n$, we identify this critical region and establish a limit theorem for the Betti number. Combined with Theorems \ref{upper} and \ref{middle}, this proves a sharp threshold for vanishing cohomology for all possible $p_i$.

\begin{thm} \label{bettiboundary}
Let $X \sim X(n,p_1, p_2, \ldots)$ with $$p_i = \left(\rho_1 \log n + \rho_2 \log \log n + c \right)^{\beta_i}n^{-\alpha_i}$$ such that $$\rho_1 = k - \sum_1^{k-1} \alpha_i \binom{k}{i+1} \textnormal{, } \rho_2 = \sum_1^{k-1} \beta_i \binom{k}{i+1} \textnormal{ and }\sum_1^k \alpha_i \binom{k}{i} = 1 = \sum_1^k \beta_i \binom{k}{i}.$$ Then $\beta^{k-1}$ the $(k-1)$-th Betti number approaches a Poisson distribution
\begin{equation*}
\beta^{k-1} \rightarrow \Poi (\mu)
\end{equation*}
with mean
\begin{equation*}
\mu = \frac{\rho_1 ^{\rho_2} e^{-c}}{k!}.
\end{equation*}

\end{thm}

We also provide a lower bound for the transition to nontrivial homology. This bound, combined with the second part of Theorem \ref{middle}, is shown to be sharp when $p_k \neq 1$.
\begin{thm}
\label{lower}
Let $X \sim X(n, p_1, p_2, \ldots)$ with $p_i = n^{-\alpha_i}$ and $\alpha_i \geq 0$ for all $i$.  If
\begin{equation} \label{noboundary}
k+1 < \sum_1^{k-1} \alpha_i \binom{k+1}{i+1}
\end{equation}
then w.h.p.\ $H_{k-1} (X, \mathbb{Z}) = 0$.
\end{thm}

The proof of Theorem \ref{upper} is handled in Sections 3 and 4. The inequality \eqref{fofree} precisely ensures every $(k-1)$-simplex of $X$ is w.h.p.\ contained in a $k$-simplex, so no single face generates a non-trivial cocyle in $H^{k-1}(X)$. With this condition satisfied we prove the result by applying \cite[Theorem 2.1]{garland}, a result connecting spectral gap theory and the homology of simplicial complexes and presented in Section 2. Most of the work lies in showing the various hypotheses of the theorem are met by our complexes, for which we use \cite[Theorem 1.1]{spec}, a tool for bounding the spectral gap of Erd\H{o}s-R\'enyi random graphs.

Theorem \ref{middle} is proven in Sections 5 and 6. The statement for the range defined by \eqref{midrange1} and \eqref{midrange2} is shown by exhibiting that our complex will have far more $(k-1)$-dimensional faces than those in adjacent dimensions, so the kernel of the coboundary map is very large. In fact, the second moment argument used in the proof yields the stronger result that within this range of values our Betti number $\beta^{k-1}$ will grow polynomially in $n$. The result when \eqref{midrange1} and \eqref{midrange3} hold follows from the model used in Proposition \ref{1stlemma}: showing our complex will w.h.p.\ contain certain subcomplexes that generate nontrivial homological cycles.

In Section 7 we prove Theorem \ref{bettiboundary}. The hypotheses of the theorem articulate a range where we have a non-zero but finite expected number of $(k-1)$-faces not contained in a $k$-face. A factorial moment argument shows this number approaches a limiting distribution, a slight adaptation of the work in Section 4 then proves these faces generate the only nontrivial cocycles of dimension $k-1$.

Finally, the proof of Theorem \ref{lower} is found in Section 8. The subset \eqref{noboundary} defines when our complex will w.h.p.\ not contain the boundary of a $k$-simplex. We show this is the most likely subcomplex to appear in $X$ that generates a $(k-1)$-cycle. Thus, when $X$ w.h.p.\ does not contain the boundary of a $k$-simplex it will have no $(k-1)$-cycles.
\end{subsection}

\begin{subsection}{Discussion}
Primarily our results concern when $p_i = n^{- \alpha_i}$ with $\alpha_i \geq 0$ or $p_i = 0$ (here we say $\alpha_i = \infty$). This was done to make the theorem statements as concise as possible. Our threshold results extend easily to when $p_i$ are more varied functions of $n$. If $p_i = \omega_i n^{-\alpha_i}$ with $\omega_i(n) \rightarrow \infty$ and $\omega_i(n) = o(n^{\ep})$ for all $\ep>0$, then Theorems \ref{upper}, \ref{middle}, and \ref{lower} still hold provided the $\alpha_i$ do not lie on the boundary between two thresholds.

Our work on this multi-parameter model confirms it as the natural bridge between $X(n, p)$ and $Y_k(n,p)$. Our theorems imply all the analogous results for these complexes, when appropriate. The boundary between Theorems \ref{upper} and \ref{middle} is sharp when $p_i = n^{-\alpha_i}$, and combined with Theorem \ref{bettiboundary} establishes a sharp upper bound for vanishing cohomology that encompasses the analogous results for clique complexes \cite[Theorem1.1]{sharp} and Linial-Meshulam complexes \cite[Theorem 1.1]{randk}.

While our upper bounds are seen to be sharp, we have not fully characterized the lower threshold for vanishing homology. So long as $p_k \neq 1$ our bounds are sharp, and Kahle proved the correct bound for clique complex case in \cite{general}, but we have been unable to generalize his arguments or find another method. For now we leave this as an open problem. \\ \\
\textbf{Open Problem 1} What is the lower threshold for the vanishing of $H_{k-1} \left(X, \mathbb{Z}\right)$ when $p_k = 1$? \\ 

Noting $p_k = 1$ implies $X$ cannot contain the unfilled boundary of a $k$-simplex, the question likely reduces to understanding the smallest homological cycle that can appear in $X$. We suspect the answer is determined, perhaps uniquely, by the largest $l < k$, such that $p_l \neq 1$. Meanwhile, the Linial-Meshulam model lacks a lower threshold for vanishing $(k-1)$-homology, and $p_{k+1}= 0$ so our bounds for $H_k(Y_k(n,p))$ are sharp.

Another open problem concerns when integer homology vanishes in a specific dimension. \\ \\
\textbf{Open Problem 2} Does \eqref{fofree} in Theorem \ref{upper} imply that w.h.p.\ $H_{k-1}(X, \Z) = 0$? \\

We understand the phase transition for $H^{k-1}(X, \mathbb{Q})$ and have reason to believe our results should hold for integer homology, but our present arguments are insufficient. We note this question is also currently unsolved for $X(n,p)$.

Although $X(n,p)$ and $Y_k(n,p)$ seem like antipodal cases of $X\left(n, p_1, p_2, \ldots \right)$, they do not fully characterize our model. We often observe asymptotic behavior dramatically different from either one. In fact, for any fixed integer $l$ we can find some $k$ such that the range of values for $p_i$ defined by applications of Theorem \ref{middle} in dimensions $k$ through $k+l$ is nontrivial. This yields a result exemplifying the differences in this model.
\begin{cor} 
Let $X \sim X(n, p_1, p_2, \ldots)$ with $p_i = n^{-\alpha_i}$, for any integer $l$ there exists an integer $k$ and an open set of $\alpha_i$ for which $X$ w.h.p. has non-trivial cohomology in dimensions $k$ through $k+l$.
\end{cor}
\end{subsection}

\end{section}

\begin{section}{Topological Preliminaries}
\begin{subsection}{Basic definitions}
Before proceeding, we lay out the definitions and theorems critical to our work. For further reference, we direct the reader to \cite{hatcher2002algebraic}.

Essentially, the homology and cohomology of a topological space is a measure of the number of ``holes" of a specific dimension in the space. Fixing some simplicial complex $X$, for any $k \geq 0$ we define $C^k(X)$ to be the vector space of linear $\mathbb{Q}$-valued functions on the $k$-simplices of $X$. We call such functions $k$-cochains and it is not hard to see that $C^k(X)$ is generated by the characteristic functions of the individual $k$-faces of $X$. For some $(k+1)$-face $\sigma = [v_0, \ldots, v_{k+1}]$ in $X$, we define the $k$-faces $\sigma_i = [v_0, \ldots, \hat{v}_i, \ldots, v_{k+1} ]$. We then the define the \emph{$k$-th coboundary map} $\delta^k : C^k(X) \rightarrow C^{k+1}(X)$ by, for some $\phi \in C^i(X)$,
\begin{equation*}
\delta^k (\phi)(\sigma) = \sum_{i=0}^{k+1} (-1)^i \phi(\sigma_i). 
\end{equation*}
One can verify that $\delta^k \circ \delta^{k-1} = 0$, so $\textnormal{Im}(\delta^{k-1}) \subseteq \textnormal{ker}(\delta^k)$. We call a $k$-cochain $\phi$ a coboundary if $\phi \in \textnormal{Im}(\delta^{k-1})$ and a cocycle if $\phi \in \textnormal{ker}(\delta^{k})$. With this we are able to define the \emph{$k$-th rational cohomology group of $X$} to be
\begin{equation*}
H^k(X, \mathbb{Q}) = \frac{\textnormal{ker}(\delta^k)}{\textnormal{Im}(\delta^{k-1})},
\end{equation*}
and the \emph{$k$-th Betti number} $\beta^k := \textnormal{dim} \left(H^k(X, \mathbb{Q})\right)$.

The homology of $X$ is defined in a similar fashion. We fix $F$ to be $\Z$ or some field, typically $\mathbb{Q}$ or a finite field. Letting $C_k(X)$ be the $F$-vector space generated by the $k$-faces of $X$, we construct our \emph{$k$-th boundary map} $\partial_k : C_k(X) \rightarrow C_{k-1}(X)$ by, for some $\sigma \in C_k(X)$,
\begin{equation*}
\partial_k (\sigma) = \sum_{i=0}^{k} (-1)^i \sigma_i.
\end{equation*}
Then we define the \emph{integer $k$-th homology group of $X$ with $F$-coefficients} by
\begin{equation*}
H_k(X, F) = \frac{\textnormal{ker}(\partial_k)}{\textnormal{Im}(\partial_{k+1})}.
\end{equation*}

Another useful definition for our work is the link of a subcomplex. Given a simplicial complex $X$ and a $k$-dimensional simplex $\sigma$ in X, we define the \emph{link of $\sigma$ in $X$}, denoted $\textrm{lk}_X(\sigma)$, to be a new simplicial complex with vertex set corresponding to the vertices of $X$ that form an $(k+1)$-face with $\sigma$. We then construct the new simplicial complex by adding the $(l-1)$-face corresponding to a set of vertices $v_1, \ldots, v_l$ precisely when the vertices $\sigma \cup \{v_1 , \ldots , v_l \}$ comprise a $(k+l)$-face in $X$.

A simplicial complex $X$ is \emph{pure $k$-dimensional} if every face of $X$ is contained in a $k$-dimensional face.

Finally, let $G$ be some graph with ordered vertices, with $D$ and $A$ the associated degree and adjacency matrices of $G$, respectively. We then construct the \emph{normalized Laplacian of $G$}, denoted $\mathcal{L}$, by
\begin{equation*}
\mathcal{L} = I - D^{-1/2} A D^{-1/2}.
\end{equation*}
For our work we look at the \emph{spectral gap of $G$} (denoted $\lambda_2 [G]$), which is the absolute value of the smallest non-zero eigenvalue of the normalized Laplacian of $G$.

\end{subsection}

\begin{subsection}{Useful Theorems}
There are several established theorems we use in our work.

The Universal Coefficients Theorem provides the link between the homology and cohomology of a simplicial complex $X$, telling us $H_{k-1}(X, \mathbb{Q}) \cong H^{k-1}(X, \mathbb{Q})$. So any statement about rational homology can be extended to cohomology, and vice versa. Moreover, a $\Z$-summand of $H_k \left(X, \Z \right)$ necessarily corresponds to a $\mathbb{Q}$-summand of $H_k(X, \mathbb{Q})$. Within our work, the language of a theorem statement primarily corresponds to whichever group we worked with in the proof. Finally, we note that the vanishing of integer homology is a much stronger statement than the vanishing of rational homology.

With the definitions established we introduce the first of the two theorems instrumental in our proof of Theorem \ref{upper}. We use a special case of Theorem 2.1 in a paper by Ballmann and \' Swi\c{a}tkowski \cite{garland}. \\ \\ 
\textbf{Cohomology Vanishing Theorem.} \cite[Theorem 2.1]{garland} Let $X$ be a pure $D$-dimensional finite simplicial complex such that for every (D-2)-dimensional face $\sigma$, the link $\textnormal{lk}_X (\sigma)$ is connected and has spectral gap
\begin{equation*}
\lambda_2 [\textnormal{lk}_X (\sigma) ] > 1- \frac{1}{D}
\end{equation*}
Then $H^{D-1} (X , \mathbb{Q}) = 0$. \\ 

We note that since $X$ is stipulated to be pure $D$-dimensional, the link of any $(D-2)$-face will be of dimension $1$. The spectral gaps of these link complexes are therefore well-defined. 

To produce the necessary estimates on these gaps we then need the help of the main result in \cite{spec}, established by Hoffman, Kahle, and Paquette. We present it here as a concise statement sufficient for our needs, noting the actual result yields more general and precise results. \\ \\
\textbf{Spectral Gap Theorem.} \cite[Theorem 1.1]{spec} Fix a $\delta > 0$ and let $G \sim G(n, p)$ with $p \geq \frac{(1+\delta) \log n}{n}$. Then G is connected and
\begin{equation*}
\lambda_2 (G) > 1 -o(1)
\end{equation*}
with probability $1 - o(n^{-\delta})$.

\end{subsection}

\end{section}

\begin{section}{Calculating free faces}
We call a $(k-1)$-face in a simplicial complex \emph{free} if it is not contained in any $k$-simplex. These subcomplexes naturally play an important role in homology, their characteristic functions generate $(k-1)$-cocycles. We let $N_{k-1}$ denote the number of free $(k-1)$-faces in $X$. Recall our complex has vertex set $[n]$, we use $j \in \binom{[n]}{k}$ to denote a set of $k$ vertices of $[n]$. Letting $C_j$ be the event that the vertices of $j$ span a free $(k-1)$-simplex, it follows that
\begin{equation*}
N_{k-1} = \sum_{j \in \binom{[n]}{k}} 1_{C_j}.
\end{equation*}

\begin{lemma}
For any $j \in \binom{[n]}{k}$,
\begin{equation}
\P \left[ C_j \right] = \left( \prod_{i=1}^{k-1} p_i^{\binom{k}{i+1}} \right) \left(1 - \prod_{i=1}^{k} p_i^{\binom{k}{i}} \right)^{n-k}. \label{eqn:1}
\end{equation}
\end{lemma}
\begin{proof}
The left parenthetical calculates the probability that $j$ is in our complex. For any $1 \leq i \leq k-1$ we need the $\binom{k}{i+1}$ possible $i$-faces on the vertices of $j$ to be contained in $X$. Proceeding inductively, the $(i-1)$-skeleton of each face is already contained in $X$ and each $i$-simplex is added independently with probability $p_i$. The right parenthetical calculates the probability these vertices do not form a $k$-simplex with one of the other $n-k$ vertices. For a fixed vertex $v$, this happens when every face of dimension $1, \ldots, k$ involving $v$ and vertices of $j$ is contained in our complex. This event that we wish to avoid occurs independently for each vertex with probability $ \prod_{i=1}^{k} p_i^{\binom{k}{i}}$, and our result follows.
\end{proof}

We now establish the threshold where these subcomplexes do not appear in our complex.

\begin{lemma} \label{nofree}
Let $X \sim X(n, p_1, p_2, \ldots)$ with $p_i = n^{-\alpha_i}$, if $$ \sum_1^k \binom{k}{i} \alpha_i < 1$$ then $X$ w.h.p.\ contains no free $(k-1)$-faces.
\end{lemma}
\begin{proof}
Recall $N_{k-1}$ counts the free faces in $X$, by \eqref{eqn:1} and linearity of expectation we have
\begin{align*}
\mathbb{E} [ N_{k-1} ] &= \sum_{j \in \binom{[n]}{k}} \E[1_{C_j}] = \sum_{j \in \binom{[n]}{k}} \P[C_j] \\
&= \sum_{j \in \binom{[n]}{k}} \left( \prod_{i=1}^{k-1} p_i^{\binom{k}{i+1}} \right) \left(1 - \prod_{i=1}^{k} p_i^{\binom{k}{i}} \right)^{n-k}\\
&= \binom{n}{k} \left( \prod_{i=1}^{k-1} p_i^{\binom{k}{i+1}} \right) \left(1 - \prod_{i=1}^{k} p_i^{\binom{k}{i}} \right)^{n-k} \\
&\leq \frac{n^k}{k!} \left(\prod_{i=1}^{k-1} n^{-\alpha_i \binom{k}{i+1}} \right) \left( e^{-(n-k) \left( \prod_{i=1}^{k} n^{- \alpha_i \binom{k}{i}} \right)} \right).
\end{align*}
Then for some $D>0$,
\begin{align*}
\E \left[ N_{k-1}\right]&\leq D \frac{n^k}{k!} \left( n^{-\sum_{i=1}^{k-1} \alpha_i \binom{k}{i+1}} \right) \left( e^{-n \left( n^{-\sum_{i=1}^{k} \alpha_i \binom{k}{i}} \right)}  \right)  \\
&= \frac{D}{k!} \left( n^{k - \sum_1^{k-1} \alpha_i \binom{k}{i+1} } \right) \left( e^{-n^{1 - \sum_1^k \alpha_i \binom{k}{i} }} \right).
\end{align*}

By hypothesis $$\sum_1^{k-1} \alpha_i \binom{k}{i} < 1,$$ so the right parenthetical of our last term is $e^{-n^\epsilon}$ for some $\epsilon > 0$. This term asymptotically dominates the rest of the expression and $\E\left[N_{k-1} \right] \rightarrow 0$ exponentially. Markov's inequality tells us $$\P \left[ N_{k-1} \geq 1 \right] \leq \E \left[ N_{k-1} \right] = o(1),$$ completing our proof.
\end{proof}

So in this regime w.h.p.\ every $(k-1)$-face of our complex is contained in a $k$-simplex, a fact necessary to utilize \cite[Theorem 2.1]{garland} and prove that $H^{k-1} (X, \mathbb{Q}) = 0$ in this range.

\end{section}

\begin{section}{Trivial Cohomology}
In this section we prove Theorem \ref{upper}, the upper threshold for vanishing cohomology, with \cite[Theorem 2.1]{garland} and \cite[Theorem 1.1]{spec} crucial to our argument.

To understand the $(k-1)$-th cohomology of a complex we need only consider its $k$-skeleton, ie.\ the subcomplex of $X$ induced by its faces of dimension $k$ and lower. We use $X_k$ to denote the $k$-skeleton of $X$, observing $H^{k-1}(X_k) = H^{k-1}(X)$. The following lemma provides the first step to invoking the \cite[Theorem 2.1]{garland}.
\begin{lemma}
Let $X \sim X(n, p_1, p_2, \ldots)$ such that $$ \sum_1^k \alpha_i \binom{k}{i} < 1$$ and $X_k$ be its $k$-skeleton. Then $X_k$ is w.h.p.\ pure $k$-dimensional.
\end{lemma}
\begin{proof}
Fixing a $1 \leq j< k-1$ we have $$\sum_1^{j+1} \alpha_i \binom{j+1}{i} \leq \sum_1^k \alpha_i \binom{k}{i} < 1,$$ so by Lemma \ref{nofree} w.h.p.\ every $j$-face of $X_k$ is contained in a $(j+1)$-simplex. Our claim follows immediately. 
\end{proof}
Thus $X_k$ satisfies the first hypothesis of \cite[Theorem 2.1]{garland}. To establish trivial cohomology we must bound the spectral gaps of the links of $X_k$.

\begin{subsection}{Using the Spectral Gap Theorem}

We wish to understand the structure of the links of the $(k-2)$-faces in our complex. Given a $(k-2)$-face $\sigma \in X_k$, we let $L_\sigma$ denote the number of vertices in $\textrm{lk}_{X_k}(\sigma)$.
\begin{lemma} \label{linklemma}
For any $(k-2)$-face $\sigma \in X$, $L_\sigma$ has the same distribution as $\textnormal{Bin} ( n- k+1, p  )$ with $ \bar{p} = \prod_1^{k-1} p_i^{\binom{k-1}{i}}$. Furthermore, $\textnormal{lk}_{X_k}(\sigma)$ has the same distribution as $G(L_\sigma, p')$ with $p' = \prod_1^k p_i^{\binom{k-1}{i-1}}$.
\end{lemma}
\begin{proof}
Fixing a $(k-2)$-face $\sigma$, a vertex $v$ will be in $lk_{X_k} (\sigma)$ if $X_k$ contains every possible simplex on $v$ and some subset of the vertices of $\sigma$. In dimension $1 \leq i \leq k-1$ there are $\binom{k-1}{i}$ such simplices, each present with probability $p_i$. Distinct vertices appearing in the link are statements about disjoint sets of simplices, so these events are independent with probability $\bar{p}$ and our statement about $L_\sigma$ follows. Similarly, the edge between two vertices of $\textnormal{lk}_{X_k}(\sigma)$ is included when $X_k$ contains every simplex of dimension $1, \ldots, k$ involving those two vertices and vertices of $\sigma$. This occurs with probability $p' = \prod_1^k p_i^{\binom{k-1}{i-1}}$, and the inclusion of distinct edges are again independent events. Thus $\textnormal{lk}_{X_k}(\sigma)$ has the same distribution as $G(L_\sigma, p')$ as desired.
\end{proof}
So the link of a $(k-2)$-face behaves like an Erd\H{o}s-R\'enyi random graph, but before applying our theorem to bound its spectral gap we must bound $L_\sigma$.

\begin{lemma} \label{linkverts}
Let $X \sim X(n, p_1, p_2 \ldots)$ with $$\sum_1^k \alpha_i \binom{k}{i} < 1,$$ then w.h.p.\ $n\bar{p}/2 \leq L_\sigma$ for every $(k-2)$-face $\sigma \in X$.
\end{lemma}
\begin{proof}
For any specific $(k-2)$-face $\sigma$ and $n$ large enough that $$n\bar{p}/2 <4 (n-k+1) \bar{p}/7,$$ Chernoff bounds give us
\begin{equation} \label{chernoff}
\P \left(L_\sigma < n \bar{p}/2 \right) \leq \P \left(L_\sigma < 4\mu/7 \right) \leq e^{-\frac{9\mu}{98}}
\end{equation}
with $\mu = (n-k+1)\bar{p}$. However, these probabilities are not independent for each $(k-2)$-face. Defining $J_\sigma$ to be the indicator random variable for $\{L_\sigma < n\bar{p}/2\}$, Markov's Inequality tells us
\begin{equation*}
\P \left[ \left( \sum_\sigma J_\sigma \right) \geq 1 \right] \leq \E \left[ \sum_\sigma J_\sigma \right] = \sum_{\sigma} \E \left[J_\sigma \right] .
\end{equation*}
There are at most $\binom{n}{k-1}$ $(k-2)$-faces in $X$ and by construction $\E[J_\sigma] = \P( L_\sigma < n\bar{p}/2)$, so
\begin{align*}
\P \left[ \left( \sum_\sigma J_\sigma \right) \geq 1 \right] &\leq \binom{n}{k-1}  \P( L_\sigma < n\bar{p}/2) \textrm{ (for some fixed $\sigma$)} \\
&\leq \binom{n}{k-1} e^{-\frac{9\mu}{98}} \textnormal{    (by \eqref{chernoff})}\\
&= \binom{n}{k-1} e^{-\frac{9(n-k+1)\bar{p}}{98}} \\
&= \binom{n}{k-1} e^{-\frac{9n\bar{p}}{98}} e^{\frac{(k-1)\bar{p}}{98}}.
\end{align*}
Since $\alpha_i \geq 0$ for all $i$, we know $$\sum_1^{k-1} \alpha_i \binom{k-1}{i} < \sum_1^{k} \alpha_i \binom{k}{i} < 1$$ and so for some $\ep >0$
\begin{equation*}
\bar{p} = \prod_1^{k-1} p_i^{\binom{k-1}{i}} = n^{-\sum_1^{k-1} \alpha_i \binom{k-1}{i}} = n^{\epsilon-1} .
\end{equation*}
Since $\frac{(k-1)\bar{p}}{98} \rightarrow 0$, we may bound $e^{\frac{(k-1)\bar{p}}{98}}$  above by a constant $C>0$. It follows that $$\P \left[ \left( \sum_\sigma J_\sigma \right) \geq 1\right] \leq C \binom{n}{k-1} e^{-\frac{9}{98} n^\epsilon } = o(1).$$
Thus w.h.p.\ $L_\sigma = 0$ for every $(k-2)$-face $\sigma$, completing our proof.
\end{proof}

We require one last lemma before proving our main result.

\begin{lemma} \label{forspec}
Let $X \sim X(n, p_1, p_2, \ldots)$ and fix $\delta>0$. If $$\sum_1^{k} \alpha_i \binom{k}{i} < 1$$ then w.h.p.
\begin{equation}
\frac{(1+\delta) \log L_\sigma}{L_\sigma} \leq p'
\end{equation}
for all $(k-2)$-faces $\sigma$ in $X$.
\end{lemma}

\begin{proof}
We let $L = n \bar{p} / 2$. Straightforward calculus shows $f(x) = \frac{(1+\delta) \log(x)}{x}$ is monotonically decreasing on $[e, \infty)$. For large $n$ we have $e < L$, so if $f(L) < p'$ then by Lemma \ref{linkverts} $f(L_\sigma) < p'$ for all $\sigma$ w.h.p. We let $\ep = 1 - \sum_1^k \alpha_i \binom{k}{i}$, noting $\ep >0$ by hypothesis. Then
\begin{align*}
\frac{f(L)}{p'} &= (1+ \delta) \frac{\log L}{Lp'} \\
&\leq (2+ 2\delta) \frac{\log n}{n \bar{p} p'} \\
&= (2+ 2\delta) \frac{\log n}{n^{1 - \sum_1^{k-1} \alpha_i \binom{k-1}{i} - \sum_1^{k} \alpha_i \binom{k-1}{i-1} } } \\
&= (2+ 2\delta) \frac{\log n}{n^{1 - \sum_1^{k} \alpha_i \binom{k}{i}} } \\
&= (2+ 2\delta) \frac{\log n}{n^{\epsilon}} \\
&= o(1).
\end{align*}
Thus w.h.p.\ $f(L_\sigma) < f(L) < p'$ for all $(k-2)$-faces $\sigma$. 
\end{proof}

\end{subsection}

\begin{subsection}{The Main Result}
We now have the machinery to prove a main theorem.

\begin{proof}[Proof of Theorem \ref{upper}]
We begin by fixing the $\delta > 0$ we will use in \cite[Theorem 1.1]{spec}: 
\begin{equation} \label{delta}
\delta = \frac{k - \sum_1^{k-2} \binom{k-1}{i+1} \alpha_i}{1-\sum_1^{k-1} \binom{k-1}{i} \alpha_i}.
\end{equation} 

A standard second moment technique, detailed in Section 6, tells us that if $f_{k-2}$ denotes the number of $(k-2)$-faces in $X$, or $X_k$, then w.h.p.
\begin{equation} \label{linkbound}
f_{k-2} \leq (1+o(1)) \E \left[ f_{k-2} \right] = (1 + o(1)) \binom{n}{k-1} \prod_1^{k-2} p_i ^{\binom{k-1}{i+1}} .
\end{equation}
By Lemma \ref{linklemma} each of these faces has link with distribution $G ( L_\sigma, p' )$, and by Lemma \ref{forspec} w.h.p.
\begin{equation*}
\frac{(1+\delta) \log L_\sigma}{L_\sigma} < p'
\end{equation*}
for all $(k-2)$-faces $\sigma$ of $X$. Thus by \cite[Theorem 1.1]{spec} the probability $P_\sigma$ that $$\lambda_2 [ \textrm{lk}_{X_k} (\sigma) ] < 1 - 1/k$$ is $o(L_\sigma^{-\delta})$. Letting $P_X$ denote the probability there exists any $(k-2)$-face whose link in $X_k$ has spectral gap less than $1-1/k$, we apply a union bound to see
\begin{align*}
P_X &\leq \sum_\sigma P_\sigma \\
&= \sum_\sigma o(L_\sigma^{-\delta}) \\
&\leq \sum_\sigma o((\frac{n \bar{p}}{2})^{-\delta}).
\end{align*}
The last line holds since w.h.p.\ $ n\bar{p}/2 <L_\sigma $, so $L_\sigma^{-\delta} < (n \bar{p}/2)^{-\delta}$. By \eqref{linkbound},
\begin{align*}
P_X &\leq \left(1 + o(1) \right) \binom{n}{k-1} \left(\prod_1^{k-2} p_i ^{\binom{k-1}{i+1}} \right) o \left(2^{\delta} (n\bar{p})^{-\delta} \right) \\
&\leq \left(1+ o(1) \right) \frac{n^{k-1}}{(k-1)!} \left(\prod_1^{k-2} p_i ^{\binom{k-1}{i+1}} \right) o \left(2^\delta (n\bar{p})^{-\delta} \right) \\
&= O \left( 2^\delta n^{k-1} n^{- \sum_1^{k-2} \alpha_i \binom{k-1}{i+1} } \left(n \cdot n^{-\sum_1^{k-1} \alpha_i \binom{k-1}{i} } \right)^{-\delta}         \right) \\
&= O \left( 2^\delta  n^{k-1- \sum_1^{k-2} \alpha_i \binom{k-1}{i+1}  } n^{-\delta \left(1 - \sum_1^{k-1} \alpha_i \binom{k-1}{i}  \right)}       \right).
\end{align*}

By our choice of $\delta$ in \eqref{delta},
\begin{align*}
P_X &=  O \left(  n^{k-1- \sum_1^{k-2} \alpha_i \binom{k-1}{i+1} } n^{- (   k- \sum_1^{k-2} \alpha_i \binom{k-1}{i+1}  )}      \right) \\
&= O\left(n^{-1} \right) \\
&= o(1).
\end{align*}

Thus w.h.p.\ $$\lambda_2 [\textnormal{lk}_{X_k}(\sigma)] > 1 - \frac{1}{k}$$ for every $(k-2)$-face $\sigma$ of X. Combining this with Lemma \ref{nofree}, we may apply \cite[Theorem 2.1]{garland} on $X_k$ to conclude that w.h.p.\ $H^{k-1}(X_k, \mathbb{Q}) = 0$. Noting that $H^{k-1}(X_k , \mathbb{Q}) \cong H^{k-1}(X, \mathbb{Q} )$ completes our proof.
\end{proof}

\end{subsection}

\end{section}

\begin{section}{Nontrivial Homology: Boundaries of Simplices}
In this Section we consider the case $$1 \leq \sum_1^k \alpha_i \binom{k}{i},\ \sum_1^{k-1} \alpha_i \binom{k+1}{i+1} < k+1, \textnormal{ and } p_k \neq 1$$ to prove the second half of Theorem \ref{middle}. 

As shown for $Y_k(n,p)$ in \cite{tophomcollapse}, the first type of homological $(k-1)$-cycle to occur in $X\left(n, p_1, p_2, \ldots \right)$ is the boundary of a $k$-dimensional simplex that is not filled in, provided such a subcomplex is possible (ie.\ $p_k \neq 1$). If $X$ contains the unfilled boundary of a $k$-face with at least one free $(k-1)$-face, then it generates a $\Z$-summand in $H_{k-1}(X,\Z)$. For a set of $k+1$ vertices $j \in \binom{[n]}{k+1}$, we define $A_j$ as the event $j$ corresponds to the unfilled boundary of a simplex with first $(k-1)$-face, determined by lexicographic order, not contained in any $k$-simplex. Letting $M_{k-1}$ denote the total number of such subcomplexes in $X$, it follows that
\begin{equation*}
M_{k-1} = \sum_{j \in \binom{[n]}{k+1}} 1_{A_j}.
\end{equation*}

We then calculate the probability of $A_j$.
\begin{lemma} For any $j\in \binom{[n]}{k+1}$,
\begin{equation}
\E [1_{A_j}] = \P[A_j]= \left( \prod_{i=1}^{k-1} p_i^{\binom{k+1}{i+1}} \right)(1-p_k) \left(1 - \prod_{i=1}^{k} p_i^{\binom{k}{i}} \right)^{n-k-1}. \label{eqn:expy}
\end{equation}
\end{lemma}
\begin{proof}
The first term calculates the probability that $X$ contains the necessary $i$-faces for $i< k$: we need every subset of $i+1$ vertices of $j$ to form an $i$-simplex. The second term is the requirement that the associated $k$-simplex is not filled in. The last term is ensuring our first $(k-1)$-face does not form a $k$-simplex with any of the remaining $n-k-1$ vertices, which occurs independently with probability $\prod_1^k p_i^{\binom{k}{i}}$ for each vertex. 
\end{proof}
We note that narrowing our consideration to when the first $(k-1)$-face is free simplifies the calculations without altering the relevant probability thresholds. 
\begin{lemma} \label{boundaries}
Let $X \sim X(n, p_1,p_2, \ldots)$ with $p_i = n^{-\alpha_i}$ and $M_{k-1}$ count the number of unfilled boundaries of $k$-simplices in $X$ with free first $(k-1)$-face. If $$1 \leq \sum_1^k \alpha_i \binom{k}{i},\ \sum_1^{k-1} \alpha_i \binom{k+1}{i+1} < k+1 , \textnormal{ and } p_k \neq 1$$ then w.h.p.\ $M_{k-1} > 0$ and $M_{k-1} \sim \E[M_{k-1}]$.
\end{lemma}
\begin{proof}
By linearity of expectation we have
\begin{align*}
\mathbb{E} [ M_{k-1} ] &= \binom{n}{k+1} \left( \prod_{1}^{k-1} p_i^{\binom{k+1}{i+1}} \right)(1-p_k) \left(1 - \prod_{1}^{k} p_i^{\binom{k}{i}} \right)^{n-k-1}  \\
&\approx \frac{1-p_k}{(k+1)!} \left( n^{k+1 - \sum_1^{k-1} \alpha_i \binom{k+1}{i+1} } \right) \left( e^{-n^{1 - \sum_1^k \alpha_i \binom{k}{i} }} \right).
\end{align*}
Our requirements on the $p_i$ imply
\begin{equation} \label{eqn:47}
\mathbb{E} [ M_{k-1} ] = \Theta \left( n^{k+1 - \sum_1^{k-1} \alpha_l \binom{k+1}{i+1} } \right),
\end{equation}
hence $\E [M_{k-1}] \rightarrow \infty$. 

The proof that this implies that $M_{k-1} \sim \E \left[M_{k-1}\right]$ is a straightforward but computationally tedious second moment argument (see e.g.\ \cite{sharp}) that can be found in Appendix A.
\end{proof}
\begin{proof}[Proof of the second part of Theorem \ref{middle}]

It follows from Lemma \ref{boundaries} that w.h.p.\ $M_{k-1} > 0$. Consider such a subcomplex $\sigma$. As the boundary of a $k$-simplex, a signed sum of its $(k-1)$-faces is in the kernel of the $(k-1)$-boundary map. Since one of these faces, $\tau$, is not contained in a $k$-simplex of $X$, no $(k-1)$-chain with a non-zero coefficient of $\tau$ can be a $(k-1)$-boundary of $X$. Thus we have a non-trivial cycle no multiple of which is a boundary, contributing a $\Z$-summand to $H_{k-1} (X , \Z)$ and a $\mathbb{Q}$-cycle to $H_{k-1}(X,\mathbb{Q})$. By the Universal Coefficients Theorem we conclude $H^{k-1} (X, \mathbb{Q}) \cong H_{k-1} (X, \mathbb{Q}) \neq 0$. 
\end{proof}

\end{section}

\begin{section}{Nontrivial Cohomology: Betti Numbers Argument}

We now consider when $$1 < \sum_1^k \alpha_i \binom{k}{i} \textnormal{ and } \sum_1^{k-1} \alpha_i \binom{k-1}{i} < 1,$$ proving the other half of Theorem \ref{middle}. 
\begin{proof}[Proof of the first part of Theorem \ref{middle}]

For $X \sim X(n, p_1, p_2, \ldots)$, with the aforementioned conditions on $p_i$, we let $f_i$ denote the number of $i$-simplices in $X$ and $\beta^i = \textrm{dim} \left( H^i(X, \mathbb{Q}) \right)$. Linear algebra considerations tell us 
\begin{equation}
f_{k-1} \geq \beta^{k-1} \geq f_{k-1} - f_k - f_{k-2}. \label{eq:3}
\end{equation}
Thus showing that w.h.p.\ $ f_{k-1} > f_k + f_{k-2}$ implies $\beta^{k-1} > 0$. We begin by calculating the expected number of faces in these dimensions:
\begin{align*}
\E \left[ f_{k-2} \right] &= \binom{n}{k-1} \prod_1^{k-2} p_i^{\binom{k-1}{i+1}} \\
\E \left[ f_{k-1} \right] &= \binom{n}{k} \prod_1^{k-1} p_i^{\binom{k}{i+1}} \\
\E\left[ f_{k} \right] &= \binom{n}{k+1} \prod_1^{k} p_i^{\binom{k+1}{i+1}}.
\end{align*}

By linearity of expectation
\begin{equation*}
\E \left[ f_{k-1} \right] \geq \E \left[\beta^{k-1} \right] \geq \E \left[ f_{k-1} \right]-\E \left[ f_{k-2} \right]-\E \left[ f_{k} \right].
\end{equation*}
Comparing the expectations in different dimensions we see
\begin{equation} \label{dominate1}
\frac{\E[ f_{k}] }{\E[ f_{k-1}]} = \frac{n-k}{k+1} \prod_1^k p_i^{\binom{k+1}{i+1} - \binom{k}{i+1}} = \frac{n-k}{k+1} \prod_1^k p_i^{\binom{k}{i}} \leq  n \prod_1^k p_i^{\binom{k}{i}} = o(1) ,
\end{equation}
because $$\prod_1^k p_i^{\binom{k}{i}} < n^{-1}$$ by hypothesis. Similarly, since $$\prod_1^{k-1} p_i^{\binom{k-1}{i}} = n^{c-1}$$ for some $c>0$, we have
\begin{equation} \label{dominate2}
\frac{\E[ f_{k-2}] }{\E[ f_{k-1}]} = \frac{k}{n-k+1} \prod_1^{k-1} p_i^{\binom{k-1}{i+1} - \binom{k}{i+1}} = \frac{k}{n-k+1} \prod_1^{k-1} p_i^{-\binom{k-1}{i}} = \frac{k n^{1-c}}{n-k+1}  = o(1).
\end{equation} 
Thus $\E \left[ f_{k-1} \right]$ asymptotically dominates the other two terms.

Before proceeding we will introduce some notation. We write $X \sim Y$ with high probability if for all $\epsilon > 0$, we have
\begin{equation*}
\lim_{n \rightarrow \infty} \P [ (1-\epsilon) \leq Y/X \leq (1+ \epsilon)] \rightarrow 1.
\end{equation*}
Letting $$\tilde{f}_{k-1} := f_{k-1} - f_k - f_{k-2},$$ it follows from \eqref{dominate1} and \eqref{dominate2} that
\begin{equation}
\E [\tilde{f}_{k-1} ] \sim \E [ \beta^{k-1}] \sim \E[ f_{k-1}]. \label{eq:4}
\end{equation} 

To prove stronger statements about $\beta^{k-1}$ we again make use of Chebyshev's Inequality. That is, if $Z$ is a random variable with $\E \left[Z\right] \rightarrow \infty$ and $\textnormal{Var}\left[Z\right] = o \left(\E[Z]^2 \right)$, then w.h.p.\ $Z \sim \E \left[Z \right]$.

Now
\begin{align*}
\textnormal{Var} \left[f_{k-1} \right] &= \E \left[f_{k-1}^2 \right] - \E \left[f_{k-1} \right]^2 \\
&= \E \left[f_{k-1}^2 \right] - \binom{n}{k}^2   \left( \prod_1^{k-1} p_i^{2\binom{k}{i+1}} \right).
\end{align*}

It remains to calculate $ \E[f_{k-1}^2]$. For any $j \in \binom{[n]}{k}$ let $E_j$  be the event that the vertices of $j$ span a $(k-1)$-face in $X$. Then
\begin{align*}
\E \left[ f_{k-1}^2 \right] &= \sum_{j,l \in \binom{[n]}{k}} \P[E_j \cap E_l] = \binom{n}{k} \sum_{l \in \binom{[n]}{k}} \P[E_j \cap E_l].
\end{align*}
The second equality follows by symmetry and fixing some set of vertices $j$, say $\{ 1, \ldots, k \}$. We proceed by grouping the $l$ according to $\abs{j \cap l}$. Through this approach we see
\begin{align*}
\E[f_{k-1}^2] &= \binom{n}{k} \sum_{l \in \binom{[n]}{k}} \P[A_j \cap A_l] \\
&= \binom{n}{k} \sum_{m=0}^k \binom{k}{m} \binom{n-k}{k-m} \left(  \prod_{i=1}^{k-1} p_i^{2\binom{k}{i+1} - \binom{m}{i+1}}  \right) \\
&= \binom{n}{k} \prod_{i=1}^{k-1} p_i^{2\binom{k}{i+1}} \left( \sum_{m=0}^k \binom{k}{m} \binom{n-k}{k-m} \prod_{i=1}^{m-1}p_i^{- \binom{m}{i+1}} \right).
\end{align*}
We pull the $m=0$ term out of the summation and use $\binom{n-k}{k} < \binom{n}{k}$ to see
\begin{equation*}
\E[f_{k-1}^2] \leq \E[f_{k-1}]^2 + \binom{n}{k} \prod_{1}^{k-1} p_i^{2\binom{k}{i+1}} \left( \sum_{m=1}^k \binom{k}{m} \binom{n-k}{k-m} \prod_1^{m-1}p_i^{- \binom{m}{i+1}} \right).
\end{equation*}

We observe
\begin{align*}
\frac{\textrm{Var} [f_{k-1}]}{\E[f_{k-1}]^2} &\leq \frac{\binom{n}{k} \prod_{1}^{k-1} p_i^{2\binom{k}{i+1}} \left( \sum_{m=1}^k \binom{k}{m} \binom{n-k}{k-m} \prod_1^{m-1}p_i^{- \binom{m}{i+1}} \right)}{\binom{n}{k}^2   \left( \prod_1^{k-1} p_i^{2\binom{k}{i+1}} \right)} \\
&= \frac{ \sum_{m=1}^k \binom{k}{m} \binom{n-k}{k-m} \prod_1^{m-1}p_i^{- \binom{m}{i+1}} }{\binom{n}{k}} \\
&= \sum_{m=1}^k O \left( n^{-m} \prod_1^{m-1}p_i^{- \binom{m}{i+1}} \right) \\
&= o(1).
\end{align*}
The final line holds from our hypotheses since $$\sum_1^{m-1} \alpha_i \binom{m}{i+1} \leq \frac{m}{k} \left( \sum_1^{k-1} \alpha_i \binom{k}{i+1} \right) \leq \frac{m}{k}\left( k \cdot \sum_1^{k-1} \alpha_i \binom{k-1}{i}  \right)< m,$$ so for $m=1, \ldots, k$, $$\prod_1^{m-1}p_i^{- \binom{m}{i+1}} = n^{\sum_1^{m-1} \alpha_i \binom{m}{i+1}} = o \left( n^m \right).$$ We conclude $f_{k-1} \sim \E[f_{k-1}]$.

We note that nothing in the above argument is unique to $f_{k-1}$, so w.h.p.\ $-f_{k-2} \sim \E \left[-f_{k-2} \right]$ and $-f_k \sim \E \left[-f_k \right]$. By linearity of expectation $\tilde{f}_{k-1} \sim \E[\tilde{f}_{k-1}]$, then from \eqref{eq:3} and \eqref{eq:4} we conclude that w.h.p.\ $\beta^{k-1} \sim \E [\beta^{k-1}] \sim f_{k-1}$. Thus $\beta^{k-1} = \textnormal{dim} \left( H^{k-1} (X , \mathbb{Q} ) \right) \neq 0$ w.h.p., which completes our proof.
\end{proof}
In fact, under these conditions we have proven a stronger result than nontrivial homology.
\begin{lemma}
Let $X \sim X(n,p_1, p_2, \ldots)$ with $p_i = n^{-\alpha_i}$ and $\alpha_i \geq 0$ for all $i$, $f_{k-1}$ count the number of $(k-1)$-faces of $X$, and $\beta^{k-1}$ be the $(k-1)$-th Betti number. If
\begin{equation*}
\sum_{i=1}^{k-1} \alpha_i \binom{k-1}{i} < 1 < \sum_{i=1}^{k} \alpha_i \binom{k}{i} ,
\end{equation*}
then w.h.p.\ $f_{k-1} \sim \beta^{k-1}$.
\end{lemma}

Our proof also shows that allowing $$\sum_1^{k} \alpha_i \binom{k}{i} = 1$$ still ensures nontrivial cohomology.
\begin{lemma}
If $$\sum_{i=1}^{k} \alpha_i \binom{k}{i} = 1,$$ then w.h.p. $H^{k-1} (X , \mathbb{Q}) \neq 0$ and $$\beta^{k-1} \geq \left(\frac{k}{k+1} \right) f_{k-1}.$$
\end{lemma}
\begin{proof}
We first calculate
\begin{align*}
\frac{\E[f_k]}{\E[f_{k-1}]} &= \frac{n-k}{k+1} \prod_1^k p_i^{\binom{k}{i}} \\
&= \frac{n-k}{n(k+1)} \\
&\approx \frac{1}{k+1}.
\end{align*}
The machinery established in the previous section then does the work for us. Since $\beta^{k-1}$ is bounded between $f_{k-1}$ and $f_{k-1} - f_k - f_{k-2}$, with $f_{k-1} \sim \E[f_{k-1}]$ and $(f_{k-1} - f_k - f_{k-2}) \sim \E[(f_{k-1} - f_k - f_{k-2})] \sim \left(\frac{k}{k+1} \right) \E[f_{k-1}]$, our result follows immediately.
\end{proof}

\end{section}

\begin{section}{Behavior at the Boundary}
In this section we explore the behavior of the $(k-1)$-th cohomology of $X(n, p_1, p_2, \ldots)$ at the upper threshold line. Specifically, we refine the parameters of our $p_i$ to elicit some interesting behavior and prove Theorem \ref{bettiboundary}.

\begin{subsection}{Free faces}
To get the threshold for free faces, and thus trivial cohomology, we must slightly refine our model. Unfortunately there is no concise way to categorize these $p_i$. We consider when $$p_i  = \left(\rho_1 \log n + \rho_2 \log \log n + c\right)^{\beta_i} n^{-\alpha_i}$$ for some constants $\beta_i, \rho_1, \rho_2$, and $c$, with $$\sum_1^k \alpha_i \binom{k}{i} = 1.$$ It follows that
\begin{align*}
\mathbb{E} [ N_{k-1} ] &\approx \frac{n^k}{k!} \left(\prod_{i=1}^{k-1} p_i^{ \binom{k}{i+1}} \right) \left( e^{-n \left( \prod_{i=1}^{k} p_i^{\binom{k}{i}} \right)} \right) \\
&= \frac{n^{k-\sum_1^{k-1} \alpha_i \binom{k}{i+1}}}{k!} \left[ \prod_1^{k-1} \left(  \left(\rho_1 + o(1)\right) \log n  \right)^{\beta_i \binom{k}{i+1}} \right] e^{- \prod_1^k \left( \rho_1 \log n + \rho_2 \log \log n + c\right)^{\beta_i \binom{k}{i}}} \\
&= \frac{n^{k-\sum_1^{k-1} \alpha_i \binom{k}{i+1}}}{k!} \left(  \left(\rho_1 + o(1)\right) \log n  \right)^{\sum_1^{k-1} \beta_i \binom{k}{i+1}}  e^{- \left( \rho_1 \log n + \rho_2 \log \log n + c\right)^{\sum_1^k \beta_i \binom{k}{i}}  }.
\end{align*}
Letting $$\sum_1^k \beta_i \binom{k}{i} = 1,$$ we have
\begin{align*}
\E [ N_{k-1}] &\approx \frac{n^{k-\sum_1^{k-1} \alpha_i \binom{k}{i+1}}}{k!}  \left(  \left(\rho_1 + o(1)\right) \log n  \right)^{\sum_1^{k-1} \beta_i \binom{k}{i+1}} e^{-(\rho_1 \log n + \rho_2 \log \log n + c  )} \\
&= \frac{n^{k-\sum_1^{k-1} \alpha_i \binom{k}{i+1}}}{k!}  \left(  \left(\rho_1 + o(1)\right) \log n  \right)^{\sum_1^{k-1} \beta_i \binom{k}{i+1}} n^{-\rho_1} (\log n)^{-\rho_2} e^{-c}.
\end{align*}
If we set $$\rho_1 = k - \sum_1^{k-1} \alpha_i \binom{k}{i+1}$$ and $$\rho_2 = \sum_1^{k-1} \beta_i \binom{k}{i+1},$$ then
\begin{equation}
\mathbb{E} [ N_{k-1} ] \rightarrow \frac{\rho_1^{\rho_2} e^{-c}}{k!}
\end{equation}
as $n \rightarrow \infty$. We then establish the following result.

\begin{lemma} \label{boundary1}
Let $X \sim X(n,p_1, p_2, \ldots)$ with $$p_i = \left(\rho_1 \log n + \rho_2 \log \log n + c \right)^{\beta_i}n^{-\alpha_i}$$ such that $$\rho_1 = k - \sum_1^{k-1} \alpha_i \binom{k}{i+1}$$ and $$\rho_2 = \sum_1^{k-1} \beta_i \binom{k}{i+1}.$$ If $$\sum_1^k \alpha_i \binom{k}{i} = 1 = \sum_1^k \beta_i \binom{k}{i},$$ then $N_{k-1}$ the number of free $(k-1)$-faces in $X$ approaches a Poisson distribution
\begin{equation*}
N_{k-1} \rightarrow \Poi (\mu)
\end{equation*}
with mean
\begin{equation*}
\mu = \frac{\rho_1^{\rho_2} e^{-c}}{k!}.
\end{equation*}
\end{lemma}
\begin{proof}
We prove this with a tedious and fairly standard factorial moment argument, found in Appendix B.
\end{proof}

\end{subsection}

\begin{subsection}{Betti Numbers}
At criticality, if we condition on the event $N_{k-1} =0$ then slightly modifying our arguments in Section 4 will show $H^{k-1}(X, \mathbb{Q}) = 0$ w.h.p. This enables us to use the limiting distribution of $N_{k-1}$ to prove an identical result for $\beta^{k-1}$.

\begin{proof}[Proof of Theorem \ref{bettiboundary}]
From Lemma \ref{boundary1} we know that given these hypotheses $N_{k-1} \rightarrow \Poi(\mu)$. We suppose $N_{k-1} = m$ for some $m \in \Z$. The characteristic functions of these $m$ free faces are $(k-1)$-cocycles. We show these cocycles are not coboundaries, and in fact constitute the only cohomological cocycles of dimension $k-1$ in $X$.

We label these faces $\sigma_1, \ldots, \sigma_m$ and their respective characteristic functions $\phi_1, \ldots, \phi_m$. Letting $R_{k-2}$ count the number of $(k-2)$-faces of $X$ contained in $m$ or fewer $(k-1)$-faces, we have
\begin{align*}
\E [R_{k-2}] &= \binom{n}{k-1} \prod_{i=1}^{k-2} p_i^{\binom{k-1}{i+1}} \left( \sum_{j=0}^m \binom{n-k+1}{j} \left(  \prod_{i=1}^{k-1} p_i^{\binom{k-1}{i}}\right)^j   \left(1-  \prod_{i=1}^{k-1} p_i^{\binom{k-1}{i}}       \right)^{n-k+1-j} \right) \\
&= o(e^{-n^{-\ep}}) \textrm{ for some $\ep >0$.}
\end{align*}
This holds since by our hypotheses $$ n \left(\prod_1^{k-1} p_i^{\binom{n-1}{i}} \right) \rightarrow \infty,$$ so the right-most term is exponentially decaying and dominates the expression.

Therefore w.h.p.\ $X$ contains no $(k-2)$-face contained solely in some combination of our $\sigma_i$. We now suppose there exists some $(k-2)$-cochain $\lambda$ such that $\delta^{k-2}(\lambda) = \sum_1^m a_i \phi_i$ with $a_i \neq 0$ for some $i$. It follows that $\lambda$ is not a $(k-2)$-coboundary. We now consider the subcomplex $X' = X- \{\sigma_1, \ldots, \sigma_m\}$, and observe $R_{k-2} = 0$ implies that $X'$ no free $(k-2)$-faces. Since $\sum_1^{k-1} \alpha_i \binom{k-1}{i} < 1$, it follows Theorem \ref{upper} that w.h.p.\ $H^{k-2} (X', \mathbb{Q}) = 0$. But $\delta^{k-2} ( \lambda) = 0$ in $X'$ and $\lambda$ isn't a coboundary in $X$ or $X'$, yielding a contradiction. Therefore no such $\lambda$ exists and we conclude each $\phi_i$ generates a unique nontrivial cocycle in $H^{k-1} \left(X, \mathbb{Q} \right)$.

To show these cochains are the only contributors to cohomology we again consider $X'$. By construction $X'$ has no free $(k-1)$-faces, and a reworking of our proof of Theorem \ref{upper} (primarily refining our estimate in Lemma \ref{linkverts} to show Lemma \ref{forspec} still holds) tells us $H^{k-1}(X', \mathbb{Q}) = 0$ w.h.p. It follows that $H^{k-1} \left(X, \mathbb{Q} \right) \cong \mathbb{Q}^m$.
\end{proof}

Implicit in our proof is the result that when $$\sum_1^k \alpha_i \binom{k}{i} =1,$$ the presence of free $(k-1)$-faces is a necessary and sufficient condition for $H^{k-1} (X, \mathbb{Q}) \neq 0$.

\end{subsection}

\end{section}

\begin{section}{Trivial Homology: A Lower Bound}
In this section we prove Theorem \ref{lower}. The requirement $$k+1 < \sum_1^{k-1} \alpha_i \binom{k+1}{i+1}$$ is exactly the condition that our complex will w.h.p.\ not contain the boundary of a $k$-simplex. Logic dictates that, as the first $(k-1)$-cycle to appear, the threshold for the presence of these subcomplexes should provide a lower bound for trivial homology. We proceed by verifying this intuition, using the fact that minimal homological cycles have bounded vertex support. After establishing these points we may apply a union bound to conclude our result.
\begin{subsection}{Cycles of small vertex support}
We begin with a few definitions identical to those in Section 5 of \cite{general}. For a $(k-1)$-chain $C$ the \emph{support} of $C$ is the union of $(k-1)$-faces with non-zero coefficients in $C$, while the \emph{vertex support} is the underlying vertex set of the support. A pure $(k-1)$-dimensional subcomplex $K$ is \emph{strongly connected} if every pair of $(k-1)$-faces $\sigma , \tau \in K^{k-1}$ can be connected by a sequence of faces $\sigma = \sigma_0 , \sigma_1, \ldots, \sigma_j = \tau$ such that $\textnormal{dim}(\sigma_i \cap \sigma_{i+1}) = k-2$ for $0 \leq i \leq j-1$. Every $(k-1)$-cycle is a linear combination of $(k-1)$-cycles with strongly connected support.
\begin{lemma} \label{vsupp}
Let $$1 < \sum_{i=1}^{k-1} \alpha_i \binom{k-1}{i} $$ and fix $D$ such that $$\frac{k- \sum_1^{k-1} \alpha_i \binom{k}{i+1}}{ \sum_1^{k-1} \alpha_i \binom{k-1}{i} - 1} < D.$$ Then w.h.p.\ all strongly connected pure $(k-1)$-dimensional subcomplexes of $X$ have fewer than $D+k$ vertices in their support.
\end{lemma}
\begin{proof}
Let $K$ be such a subcomplex, since it is strongly connected we may order its faces $f_1, f_2, \ldots f_m$ where each face $f_j$, for $j>1$, has $(k-2)$-dimensional intersection with at least one $f_l$ with $l < j$. This induces an ordering on the supporting vertices $v_1, \ldots , v_s$ by looking at the vertex supports of $f_1, f_1 \cup f_2, f_1 \cup f_2 \cup f_3, \ldots$ Thus each vertex after $v_k$ corresponds to the addition of a $(k-1)$-face $f_j$, along with the $\binom{k-1}{i}$ $i$-dimensional faces of $f_j$ that include this vertex (and hence were not contained in $f_1 \cup \cdots \cup f_{j-1}$
). 

If $K$ has $D+k$ vertices, it follows that there are at least $$\binom{k}{i+1} + D \binom{k-1}{i}$$ $i$-dimesional faces for each $1 \leq i \leq k-1$. Now either $X$ w.h.p.\ contains no $(k-1)$-simplices, in which case the result is trivial, or $$\prod_{i=1}^{k-1} p_i^{\binom{k}{i+1}} = \prod_{i=1}^{k-1} n^{-\alpha_i \binom{k}{i+1}} = n^{-k + \beta}$$ for some $\beta >0$. By hypothesis $$\prod_{i=1}^{k-1} p_i^{\binom{k-1}{i}} = \prod_{i=1}^{k-1} n^{-\alpha_i \binom{k-1}{i}} = n^{-1 - \epsilon}$$ for some $\epsilon >0$. Choosing $D$ such that $\beta < D \epsilon$, we apply a union bound on the probability of $X$ containing a subcomplex isomorphic to $K$:
\begin{align*}
\P(\exists \textrm{ subcomplex}) &\leq (D+k)! \binom{n}{D+k} \prod_1^{k-1} p_i^{\binom{k}{i+1} + D \binom{k-1}{i}} \\
&= (D+k)! \binom{n}{D+k} n^{(-k+\beta) - D (1 + \epsilon)} \\
&\leq n^{D+k} n^{-(D+k)} n^{\beta - D \epsilon} \\
&= n^{\beta - D\epsilon} \\
&= o(1).
\end{align*}
The last line holds by our choice of $D$. 

As there are finitely many isomorphism classes of strongly connected $(k-1)$-complexes on $N+k$ vertices, a union bound shows that w.h.p.\ none of them are subcomplexes of $X$. We complete our proof by observing that any such complex with more vertices must contain a strongly connected subcomplex on $D+k$ vertices, for example the subcomplex induced by the first $D+k$ ordered vertices.
\end{proof}
\end{subsection}

\begin{subsection}{The threshold for a simplex boundary}
Here we prove our lower threshold for vanishing homology, which is sharp when $p_k \neq 1$.
\begin{proof}[Proof of Theorem \ref{lower}]
We consider some non-trivial $(k-1)$-cycle $\gamma$ with strongly connected support and $K$, its induced subcomplex in $X$. By our hypothesis we have that $$k+1 < \sum_1^{k-1} \alpha_i \binom{k+1}{i+1},$$ and either $X$ will w.h.p.\ contain no $(k-1)$-simplices, making the result trivial, or $$\sum_1^{k-1} \alpha_i \binom{k}{i+1} < k.$$ Moreover,
\begin{align*}
\sum_1^{k-1} \alpha_i \binom{k-1}{i} &= \sum_1^{k-1} \alpha_i \frac{i+1}{k} \binom{k}{i+1} \\
&= \sum_1^{k-1} \alpha_i \frac{i+1}{k} \frac{k-i}{k+1} \binom{k+1}{i+1} \\
&\geq \frac{1}{k+1} \sum_1^{k-1} \alpha_i \binom{k+1}{i+1} > 1.
\end{align*} 
Thus we may invoke Lemma \ref{vsupp} to conclude $K$ is w.h.p.\ supported on less than $D+k$ vertices. As in that proof, we may order the vertices $v_1, \ldots, v_{k+m}$ for some $m < D$. We prove our result by removing one vertex at a time from $K$ and counting the faces containing it that must also be removed. 

Since we have a non-trivial cycle every vertex is contained in \textbf{at least $k$} $(k-1)$-simplices. Removing $v_{k+m}$ first, we observe the fewest faces are removed if $v_{k+m}$ is contained in exactly $k$ $(k-1)$-simplices. In this case we then remove $\binom{k}{i}$ $i$-dimensional faces for each $i$. We then remove vertices $v_{k+m-1}, \ldots , v_{k+1}$, and by construction each one was contained in a $(k-1)$-face comprised exclusively of vertices before it, so at each removal step we remove at least that simplex. Thus at each removal we account for at least $\binom{k-1}{i}$ $i$-faces for each $i$. The last $k$ vertices correspond to our initial $(k-1)$-simplex. Putting this together we get a lower bound on the probability of a subcomplex isomorphic to $K$ appearing:

\begin{align*}
\P(\exists \textrm{ a subcomplex}) &\leq (k+m)! \binom{n}{k+m} \left(\prod_1^{k-1} p_i^{\binom{k}{i}} \right) \left( \prod_1^{k-1} p_i^{\binom{k-1}{i}} \right)^{m-1} \left( \prod_1^{k-1} p_i^{\binom{k}{i+1}} \right) \\
&\leq n^{k+m} \left(\prod_1^{k-1} p_i^{\binom{k+1}{i+1}} \right) \left(\prod_1^{k-1} p_i^{\binom{k-1}{i}} \right)^{m-1} \\
&\leq \left( n^{k+1} \prod_1^{k-1} p_i^{\binom{k+1}{i+1}} \right) \left( n^{m-1} \prod_1^{k-1} p_i^{\binom{k-1}{i}} \right)^{m-1} \\
&= o(1).
\end{align*}
The last line holds since $$k+1 < \sum_1^{k-1} \alpha_i \binom{k+1}{i+1} \textnormal{ and } 1 < \sum_1^{k-1} \alpha_i \binom{k-1}{i}.$$ 

As there are finitely many isomorphism types of strongly connected $(k-1)$-complexes on less than $D+k$ vertices, we may apply this argument to each of them and apply a union bound to conclude that w.h.p.\ none of them are subcomplexes of $X$. Thus we w.h.p.\ have no non-trivial $(k-1)$-cycles, and $H_{k-1}(X, \Z) = 0$.
\end{proof}
\end{subsection}
\end{section}

\appendix
\begin{section}{Boundaries of Simplices}
\begin{proof}[Proof of Lemma \ref{boundaries}]
We consider the case $$1 \leq \sum_1^k \binom{k}{l} \alpha_l$$ where (from \eqref{eqn:47} in Section 5) we have that $\E [M_{k-1} ] \rightarrow \infty$. By Chebyshev's inequality, 
\begin{equation*}
\mathbb{P} \Big[ \big\lvert M_{k-1} - \E [M_{k-1}] \big\rvert \geq \E [M_{k-1}] \Big] \leq \frac{\textrm{Var}[M_{k-1}]}{\E [M_{k-1}]^2}.
\end{equation*}
Thus if we can show $\textrm{Var} [M_{k-1}] = o\left( \E [M_{k-1}]^2 \right)$, then we may conclude
\begin{equation*}
\P [ M_{k-1} > 0 ] \rightarrow 1.
\end{equation*}

Considering $M_{k-1}$ as a sum of indicator random variables, 
\begin{align*}
\textrm{Var}[M_{k-1}] &\leq \mathbb{E}[M_{k-1}] + \sum_{i,j \in \binom{[n]}{k}} \textrm{Cov}[1_{A_i}, 1_{A_j}] \\
&=\mathbb{E}[M_{k-1}] + \sum_{i,j \in \binom{[n]}{k}} \left(  \P[A_i \cap A_j] - \P[A_i]\P[A_j]  \right).
\end{align*}

Clearly $\E[M_{k-1}] = o \left( \E[M_{k-1}]^2 \right)$, to handle the sum we consider pairs $i, j \in \binom{[n]}{k+1}$ and break them into 3 cases depending on $I = \abs{i \cap j}$. To make the calculations more readable we introduce some useful notation, defining $\eta_k$ to be
\begin{equation*}
\eta_k = (1-p_k) \prod_1^{k-1} p_l^{\binom{k+1}{l+1}},
\end{equation*}
the probability that our complex contains the unfilled boundary of a specific $k$-simplex. We define $\gamma_k$ as
\begin{equation*}
\gamma_k = \prod_1^k p_l^{\binom{k}{l}},
\end{equation*}
the probability that a fixed $(k-1)$-face and vertex form a $k$-simplex.

\begin{subsection}{$I = 0$} We begin by calculating $\P[A_i \cap A_j]$. The probability that both boundaries are in our complex but unfilled is $\eta_k^2$. By inclusion-exclusion principles the probability that neither $\sigma_i$ nor $\sigma_j$, the associated first $(k-1)$-faces of these subcomplexes, form a $k$-simplex with a vertex outside of $i \cup j$ is $1 - 2\gamma_k + \gamma_k^2$, and there are $n-2k-2$ such vertices. Finally, we must have that no $k$-face is formed between one subcomplex and a single vertex of the other. While this probability can be explicitly calculated, every term that isn't 1 will contain a copy of $\gamma_k$, so this probability is $1 - O(\gamma_k)$. Thus
\begin{equation*}
\P[A_i \cap A_j] = \eta_k^2 \left(1-2\gamma_k +\gamma_k^2 \right)^{n-2k-2} \left(1- O(\gamma_k) \right),
\end{equation*}
and by \eqref{eqn:expy} in Section 5 we know
\begin{align*}
\P[A_i] \P[A_j] &= \left( \eta_k \left(1-\gamma_k\right)^{n-k-1} \right)^2 \\
&= \eta_k^2 \left(1-2\gamma_k +\gamma_k^2\right)^{n-k-1} \\
&= \eta_k^2 \left(1-2\gamma_k +\gamma_k^2 \right)^{n-2k-2} \left(1-2\gamma_k +\gamma_k^2 \right)^{k+1} \\
&= \eta_k^2 \left(1-2\gamma_k +\gamma_k^2 \right)^{n-2k-2} \left(1 - O(\gamma_k) \right).
\end{align*}

Thus $$\P[A_i \cap A_j] - \P[A_i] \P[A_j] = \eta_k^2 \left(1-2\gamma_k +\gamma_k^2 \right)^{n-2k-2} O(\gamma_k)$$ and there are $O\left(n^{2k+2} \right)$ such pairs $i,j$, so the overall contribution of these pairs to our sum is
\begin{align*}
S_0 &= O\left(n^{2k+2} \eta_k^2 \left(1-2\gamma_k +\gamma_k^2 \right)^{n-2k-2} \gamma_k \right) \\
&= O\left(n^{2k+2} \eta_k^2 \left(1-2\gamma_k +\gamma_k^2 \right)^{n-k-1} \gamma_k \right).
\end{align*}
The second equality holds by restricting our consideration to $n > k$, then $\gamma_k \leq n^{-1} < k^{-1}$ and there is some $C > 0$ such that
\begin{equation*} 
(1-2\gamma_k + \gamma_k^2)^{k+1} > (1-2\gamma_k)^{k+1} > (1-2k^{-1})^k > C,
\end{equation*} 
so removing this term does not affect our big-$O$ calculations.

Since
\begin{equation*}
\mathbb{E} [M_{k-1}]^2 = \binom{n}{k+1}^2 \eta_k^2 (1-\gamma_k)^{2(n-k-1)} = O \left( n^{2k+2} \eta_k^2 (1-2\gamma_k + \gamma_k^2)^{n-k-1} \right)
\end{equation*}
and $\gamma_k \rightarrow 0$ we conclude
\begin{equation*}
\frac{S_0}{\mathbb{E}[M_{k-1}]^2} = O(\gamma_k) = o(1).
\end{equation*}
Hence the contribution of these pairs to the variance is seen to be $o\left(\mathbb{E}[M_{k-1}]^2 \right)$.

\end{subsection}

\begin{subsection}{$I = 1$} 
The probability of both subcomplexes being in $X$ is again $\eta_k^2$ since the two don't share a face of dimension greater than 0. We again use inclusion-exclusion to calculate the probability that $\sigma_i$ and $\sigma_j$ don't form $k$-simplices with another vertex. However, these faces may or may not both contain the shared vertex: if they don't then the calculations are identical to above, so we assume the alternative. In this case the two $k$-faces formed with some new vertex would share a common edge. So the probability is $\left(1 - 2\gamma_k + \gamma_k^2 p_1^{-1}\right)$ for each of the $n-2k-1$ remaining vertices. Similarly, the probability we don't have a $k$-face consisting of $\sigma_i$ or $\sigma_j$ and a vertex in $i \bigtriangleup j$ is $1-O\left(\gamma_k p_1^{-1}\right)$. We then calculate $\P [A_i \cap A_j]$ to be
\begin{equation*}
\P [A_i \cap A_j] = \eta_k^2 \left(1- 2\gamma_k + \gamma_k^2 p_1^{-1}\right)^{n-2k-1} \left(1 - O(\gamma_k p_1^{-1})\right).
\end{equation*}

Before calculating $\P[A_i \cap A_j] - \P[A_i] \P[A_j]$, we observe
\begin{align*}
1-2\gamma_k +\gamma_k^2 &= \left(1-2\gamma_k + \gamma_k^2 p_1^{-1}\right) \frac{1-2\gamma_k +\gamma_k^2}{1-2\gamma_k + \gamma_k^2 p_1^{-1}} \\
&= \left(1-2\gamma_k + \gamma_k^2 p_1^{-1}\right) \left(1 - \frac{\gamma_k^2(p_1^{-1} -1)}{1-2\gamma_k + \gamma_k^2 p_1^{-1}} \right) \\
&= \left(1-2\gamma_k + \gamma_k^2 p_1^{-1}\right) \left(1 - O(\gamma_k^2 p_1^{-1})\right).
\end{align*}
The last equality holds by an identical argument to the one in the first case: we can bound $1-2\gamma_k+\gamma_k^2p_1^{-1}$, and consequently its inverse, from above and below by constants. We use this to calculate
\begin{align*}
\P[A_i] \P[A_j] &= \eta_k^2 \left(1-2\gamma_k +\gamma_k^2 \right)^{n-k-1} \\
&= \eta_k^2 \left[ \left(1-2\gamma_k + \gamma_k^2 p_1^{-1}\right) \left( 1 - O \left(\gamma_k^2 p_1^{-1} \right) \right)  \right]^{n-k-1} \\
&= \eta_k^2 \left(1-2\gamma_k + \gamma_k^2 p_1^{-1}\right)^{n-k-1} \left(1 - O\left(\gamma_k^2 p_1^{-1}\right) \right)^{n-k-1}.
\end{align*}
But since $\gamma_k < n^{-1}$ we have
\begin{align*}
\left(1 - O\left(\gamma_k^2 p_1^{-1}\right)\right)^{n-k-1} &= 1 - O\left( n \gamma_k^2 p_1^{-1}\right)  \\
&= 1 - O\left(\gamma_k p_1^{-1}\right).  
\end{align*}
We calculate
\begin{align*}
\P[A_i] \P[A_j] &= \eta_k^2 \left(1-2\gamma_k + \gamma_k^2 p_1^{-1}\right)^{n-k-1} \left(1 - O\left(\gamma_k p_1^{-1}\right) \right) \\
&= \eta_k^2 \left(1-2\gamma_k + \gamma_k^2 p_1^{-1}\right)^{n-2k-1} \left(1-2\gamma_k + \gamma_k^2 p_1^{-1}\right)^k \left(1 - O\left(\gamma_k p_1^{-1}\right) \right) \\
&= \eta_k^2 \left(1-2\gamma_k + \gamma_k^2 p_1^{-1}\right)^{n-2k-1} \left(1-O\left(\gamma_k\right) \right) \left(1 - O\left(\gamma_k p_1^{-1}\right) \right) \\
&= \eta_k^2 \left(1-2\gamma_k + \gamma_k^2 p_1^{-1}\right)^{n-2k-1} \left(1 - O\left(\gamma_k p_1^{-1}\right) \right).
\end{align*}

Therefore $$\P[A_i \cap A_j] - \P[A_i] \P[A_j] = \eta_k^2 \left(1-2\gamma_k + \gamma_k^2 p_1^{-1}\right)^{n-2k-1}O\left(\gamma_k p_1^{-1} \right)$$ with $O\left(n^{2k+1}\right)$ such pairs $i,j$, making the total contribution of these pairs to the variance
\begin{align*}
S_1 &= O \left(  n^{2k-1}  \eta_k^2 (1-2\gamma_k + \gamma_k^2 p_1^{-1})^{n-2k+1}\gamma_k p_1^{-1})      \right) \\
&= O \left(  n^{2k-1}  \eta_k^2 (1-2\gamma_k + \gamma_k^2 p_1^{-1})^{n-k-1}\gamma_k p_1^{-1})      \right).
\end{align*}
As before, the second equality follows from bounding $(1-2\gamma_k + \gamma_k^2 p_1^{-1})^{k-1}$ by constants on either side. 

Since $$\mathbb{E} [M_{k-1}]^2 = O \left( n^{2k+2} \eta_k^2 (1-2\gamma_k + \gamma_k^2)^{n-k-1} \right)$$ it follows that
\begin{align*}
\frac{S_1}{\mathbb{E} [M_{k-1}]^2} &= O \left(   \frac{(1-2\gamma_k + \gamma_k^2 p_1^{-1})^{n-k-1}\gamma_k p_1^{-1}}{n (1-2\gamma_k + \gamma_k^2)^{n-k-1}}    \right) \\
&= O \left( \frac{\gamma_k p_1^{-1}}{n} \left(1 + \frac{\gamma_k^2(p_1^{-1}-1)}{1-2\gamma_k +\gamma_k^2}          \right)^{n-k-1}            \right) \\
&= O \left( \frac{\gamma_k p_1^{-1}}{n} \left(1 + \frac{\gamma_k^2p_1^{-1}}{1-2\gamma_k}          \right)^{n-k-1}            \right).
\end{align*}
We proceed by bounding the right term by a constant.
\begin{align*}
\left(1 + \frac{\gamma_k^2p_1^{-1}}{1-2\gamma_k}          \right)^{n-k-1} &\leq \exp \left( (n-k-1) \frac{\gamma_k^2p_1^{-1}}{1-2\gamma_k} \right) \\
&\leq \exp \left( \frac{ n \gamma_k^2 p_1^{-1}}{1-k} \right) \\
&\leq e^{1/(1-k)}.
\end{align*}
Then
\begin{align*}
\frac{S_1}{\mathbb{E} [ M_{k-1} ] ^2} &= O \left(  \frac{\gamma_k p_1^{-1}}{n}    \right) \\
&= O\left(\frac{1}{n}\right) \\
&= o(1).
\end{align*}
Thus the contribution of these pairs is also $o\left(\mathbb{E} [ N_{k-1} ]^2 \right)$, as desired.
\end{subsection}

\begin{subsection}{$2 \leq I \leq k$}
In this final case the probability of the two subcomplexes being contained is $\eta_k^2 \eta_I^{-1}$ where $\eta_I := \prod_1^{I-1} p_l^{\binom{I}{l+1}}$. The $\eta_I^{-1}$ accounts for all  faces common to $i$ and $j$, which would otherwise be counted twice.  We note $\sigma_i$ and $\sigma_j$ share between $I-2$ and $I$ vertices, and assuming maximal overlap provides an upper bound on $\P[A_i \cap A_j]$. Hence the probability that neither will form a $k$-simplex with some other vertex is at most $\left(1 - 2\gamma_k + \gamma_k^2 \gamma_I^{-1}\right)^{n-2k-2+I}$ with $\gamma_I := \prod_1^{I} p_l^{\binom{I}{l}}$. The probability of one not forming a $k$-simplex with one vertex of the other is $1-O\left(\gamma_k \gamma_I^{-1} \right)$. We see
\begin{equation*}
\P[A_i \cap A_j] = \eta_k^2 \eta_I^{-1} \left(1 - 2\gamma_k + \gamma_k^2 \gamma_I^{-1} \right)^{n-2k-2+I} \left( 1-O\left(\gamma_k \gamma_I^{-1}\right) \right) .
\end{equation*}
Just as in the previous case,
\begin{align*}
1 - 2\gamma_k + \gamma_k^2 &= \left(1 - 2\gamma_k + \gamma_k^2 \gamma_I^{-1}\right) \left( 1 - O\left(\gamma_k^2 \gamma_I^{-1}\right) \right).
\end{align*}
We now calculate
\begin{align*}
\P[A_i] \P[A_j] &= \eta_k^2 \left(1-2\gamma_k + \gamma_k^2 \right)^{n-k-1} \\
&= \eta_k^2 \left(1 - 2\gamma_k + \gamma_k^2 \gamma_I^{-1}\right)^{n-k-1} \left( 1 - O\left( \gamma_k^2 \gamma_I^{-1}\right) \right)^{n-k-1} \\
&=  \eta_k^2 \left(1 - 2\gamma_k + \gamma_k^2 \gamma_I^{-1}\right)^{n-2k-2+I} \left( 1 - O\left(\gamma_k \gamma_I^{-1}\right) \right).
\end{align*}
It then follows that
\begin{align*}
\frac{\P[A_i] \P[A_j]}{\P[A_i \cap A_j]} &= \frac{  \eta_k^2 \left(1 - 2\gamma_k + \gamma_k^2 \gamma_I^{-1}\right)^{n-2k-2+I} \left( 1 - O\left(\gamma_k \gamma_I^{-1}\right)\right)}{ \eta_k^2 \eta_I^{-1} \left(1 - 2\gamma_k + \gamma_k^2 \gamma_I^{-1}\right)^{n-2k-2+I} \left( 1-O\left(\gamma_k \gamma_I^{-1}\right) \right)} \\
&= O\left(\eta_I\right).
\end{align*}
Thus if $\eta_I \neq 1$ then $\P[A_i \cap A_j] - \P[A_i] \P[A_j] = \left(1 - o(1)\right) \P[A_i \cap A_j]$, and otherwise $\P[A_i \cap A_j] - \P[A_i] \P[A_j] = \eta_k^2 \left(1 - 2\gamma_k + \gamma_k^2 \gamma_I^{-1}\right)^{n-2k+I} O\left(\gamma_k \gamma_I^{-1}\right)$. There are $O\left(n^{2k+2-I}\right)$ such pairs, so their total contribution to the variance is either
\begin{align*}
S_I &= O \left( n^{2k+2-I} \eta_k^2 \eta_I^{-1} \left(1 - 2\gamma_k + \gamma_k^2 \gamma_I^{-1}\right)^{n-2k-2+I} \right) \\
&= O \left( n^{2k+2-I} \eta_k^2 \eta_I^{-1} \left(1 - 2\gamma_k + \gamma_k^2 \gamma_I^{-1}\right)^{n-k-1}  \right),
\end{align*}
or
\begin{equation*}
S_I = O \left( n^{2k+2-I} \eta_k^2 \left(1 - 2\gamma_k + \gamma_k^2 \gamma_I^{-1}\right)^{n-k-1} \gamma_k \gamma_I^{-1} \right) .
\end{equation*}

In the first case we have 
\begin{align*}
\frac{S_I}{\mathbb{E} [ M_{k-1}]^2} &= O\left(\frac{n^{2k+2-I} \eta_k^2 \eta_I^{-1} \left(1 - 2\gamma_k + \gamma_k^2 \gamma_I^{-1}\right)^{n-k-1} }{n^{2k+2} \eta_k^2 \left(1-2\gamma_k + \gamma_k^2\right)^{n-k-1}}            \right) \\
&=  O\left(\frac{\eta_I^{-1}}{n^I} \left(\frac{1 - 2\gamma_k + \gamma_k^2 \gamma_I^{-1}}{1 - 2\gamma_k + \gamma_k^2}  \right)^{n-k-1}           \right).
\end{align*}
Just as before, the right-most term can be bounded above by a constant. We note $$\sum_1^{I-1} \alpha_l \binom{I}{l+1} < \frac{I}{k} \sum_1^{k-1} \alpha_l \binom{k}{l+1} < \frac{I}{k} k = I$$ and conclude
\begin{align*}
\frac{S_I}{\mathbb{E} [ M_{k-1}]^2} &= O \left( \frac{\eta_I^{-1}}{n^I} \right) \\
&= O \left( n^{-I + \sum_1^{I-1} \alpha_l \binom{I}{l+1}} \right) \\
&= o(1).
\end{align*}

In the second case we have
\begin{align*}
\frac{S_I}{\mathbb{E} [ M_{k-1}]^2} &= O\left(\frac{\gamma_k \gamma_I^{-1}}{n^I} \left(\frac{1 - 2\gamma_k + \gamma_k^2 \gamma_I^{-1}}{1 - 2\gamma_k + \gamma_k^2}  \right)^{n-k-1}       \right) \\
&= O\left(\frac{\gamma_k \gamma_I^{-1}}{n^I}\right) \\
&= O\left(n^{-I}\right) \\
&= o(1).
\end{align*}
Thus $S_I = o\left(\mathbb{E} [ M_{k-1}]^2\right)$ for $2 \leq I \leq k$. We therefore have that $\E[M_{k-1}^2] = o \left( \E[M_{k-1}]^2\right)$, and our result that $M_{k-1} \sim \E[M_{k-1}]$ follows by Chebyshev's Inequality.

\end{subsection}
\end{proof}
\end{section}

\begin{section}{Factorial Moments of Free Faces}
\begin{proof}[Proof of Lemma \ref{boundary1}]

Similar to previous second moment calculations:
\begin{align*}
\E[N_{k-1}^2] &= \binom{n}{k} \sum_{m=0}^k \binom{k}{m} \binom{n-k}{k-m}  \left( \prod_{i=1}^{k-1} p_i^{2\binom{k}{i+1} - \binom{m}{i+1}}\right) \left(1- 2\prod_1^k p_i^{\binom{k}{i}} + \prod_1^k p_i^{2\binom{k}{i} - \binom{m}{i}} \right)^{n-2k+m} \left(1- o(1)\right) \\
&\approx \binom{n}{k} \left(\prod_{i=1}^{k-1} p_i^{2\binom{k}{i+1}} \right) \sum_{m=0}^k \binom{k}{m} \binom{n-k}{k-m} \left( \prod_{i=1}^{m-1}p_i^{- \binom{m}{i+1}} \right) \left(1- 2\prod_1^k p_i^{\binom{k}{i}} + \prod_1^k p_i^{2\binom{k}{i} - \binom{m}{i}} \right)^{n-2k+m} 
\end{align*}
Pulling out the $m=0$ summand, asymptotically
\begin{align*}
 \binom{n}{k} \binom{n-k}{k} \left(\prod_{i=1}^{k-1} p_i^{2\binom{k}{i+1}} \right) \left(1- \prod_1^k p_i^{\binom{k}{i}}\right)^{2(n-2k)}  &\approx  \left(\binom{n}{k} \left(\prod_{i=1}^{k-1} p_i^{\binom{k}{i+1}} \right)  \left(1- \prod_1^k p_i^{\binom{k}{i}}\right)^{n-k} \right)^2 \\
 &= \E [ N_{k-1}]^2 .
\end{align*}
Meanwhile, the $m= k$ term is seen to be $\E[N_{k-1}]$. We claim the $k-1$ other summands do not contribute in the limit. For a fixed $m = 1, \ldots, k-1$ let $d_m < 1$ be some constant value such that
\begin{equation*}
d_m > \textnormal{max} \left\{ 1 - \frac{m(m-1)}{k(k-1)} , 1 - \frac{m - \sum_1^{m-1} \alpha_i \binom{m}{i+1}}{k - \sum_1^{k-1} \alpha_i \binom{k}{i+1}}        \right\}.
\end{equation*}
Both fraction terms are between $0$ and $1$, so such a $d_m$ exists. For sufficiently large $n$ we have
\begin{align*}
1- 2\prod_1^k p_i^{\binom{k}{i}} + \prod_1^k p_i^{2\binom{k}{i} - \binom{m}{i}} &= 1 - \left( 2 - \prod_1^k p_i^{\binom{k}{i} - \binom{m}{i}} \right)  \prod_1^k p_i^{\binom{k}{i}} \\
&\leq 1 -(1+ d_m) \prod_1^k p_i^{\binom{k}{i}}.
\end{align*}
Thus there exists a constant $D$ such that
\begin{align*}
\left(1- 2\prod_1^k p_i^{\binom{k}{i}} + \prod_1^k p_i^{2\binom{k}{i} - \binom{m}{i}} \right)^{n-2k+m} &\leq \left(1 -(1+ d_m) \prod_1^k p_i^{\binom{k}{i}} \right)^{n-2k+m} \\
&\leq D e^{-n(1+d_m) \left(\prod_1^k  p_i^{\binom{k}{i}}   \right)} \\
&= D e^{-(1+d_m) (\rho_1 \log n +\frac{k-1}{2} \log \log n + c)} \\
&= D n^{-(1+d_m) \rho_1} (\log n)^{-(1+d_m) \frac{k-1}{2}} e^{-(1+d_m)c}.
\end{align*}
Then our construction of $d_m$, $$n^{-(1+d_m) \rho_1} = o \left( n^{-2k + m + \sum \alpha_i \binom{k}{i} - \sum \alpha_i \binom{m}{i}} \right)$$ and $$ \left( \log n \right)^{-(1+d_m) \frac{k-1}{2}} = o \left( (\log n ) ^{ -(k-1) + \frac{m(m-1)}{2k}} \right).$$ It follows that the corresponding summand is bounded by
\begin{align*}
D n^{2k - m} \prod_1^{k-1} p_i^{2 \binom{k}{i+1} - \binom{m}{i+1}} n^{-(1+d_m) \rho_1} (\log n)^{-(1+d_m) \frac{k-1}{2}} = o(1),
\end{align*} 
thereby contributing nothing as $n \rightarrow \infty$. Therefore, 
\begin{equation*}
\E[\left( N_{k-1} \right)_2 ] = \E[ N_{k-1}^2] - \E[N_{k-1}] = \E[N_{k-1}]^2 (1-o(1)) \rightarrow \E[N_{k-1}]^2
\end{equation*}
as $n \rightarrow \infty$. We will now establish a similar result for each factorial moment. \\

We direct our attention to the $l$-th factorial moment of $N_{k-1}$, assuming that $\E[(N_{k-1})_{j}] \rightarrow \E[N_{k-1}]^j$ for all $j<l$. Using the notation a Section 3 we have
\begin{align*}
\E\left[N_{k-1}^l \right] &= \E \left[ \left(\sum_{\sigma \in \binom{[n]}{k}} 1_{C_\sigma} \right)^l \right] \\
&= \sum_{\sigma_1, \ldots , \sigma_l \in \binom{[n]}{k}} \P \left[ C_{\sigma_1} \cap \dots \cap C_{\sigma_l} \right].
\end{align*}

We break up this sum into two parts: where no two $\sigma_i$'s are identical and where such two $\sigma_i$ are the same. Considering the first case, an identical argument for $l=2$ tells us the only summand contributing in the limit corresponds to when no two faces share any vertices, and this term converges to $\E[N_{k-1}]^l$. 

Moving on to the second case, we let $s(l,j)$ and $S(l,j)$ denote Stirling numbers of the first and second kind, respectively. There are $S(l, j)$ ways to break our $\sigma_i$ up into $j$ groups where all faces in a group are the same. Moreover, for any such configuration into $j$ groups, the corresponding contribution to $\E[N_{k-1}^l]$ would be $\E[N_{k-1}^j]$. We begin by pulling out $S(l,l-1)=-s(l,l-1)$ copies of $\E[N_{k-1}^{l-1}]$. However, the number of partitions of $\sigma_i$ into $k-2$ groups has now been overcounted. There should only be $S(l,l-2)$ such configurations, but we have just counted $-s(l,l-1) S(l-1,l-2)$ of them, so we add $S(l,l-2) + s(l,l-1)S(l-1,l-2) = -s(l,l-2)$ copies of $\E[N_{k-1}^{l-2}]$. 

Fixing some $j< l-1$, we now assume attaching a coefficient of $-s(l,m)$ to $\E[N_{k-1}^{m}]$ for all $m > j$ ensures every partition of the $\sigma_i$ into $j+1, \ldots l-1$ sets is properly counted. Then for each $m>j$, the $-s(l,m)$ copies of $ \E[N_{k-1}^{m}]$ count $-s(l,m) S(m,j)$ partitions into just $j$ groups. Meanwhile we know there are actually only $S(l,j)$ distinct partitions, so we must add:
\begin{align*}
S(l,j) + \sum_{m=j+1}^{l-1} s(l,m) S(m,j) &= \sum_{m=j+1}^{l} s(l,m) S(m,j) \\
&=  \sum_{m=j}^{l} s(l,m) S(m,j) - s(l,j)S(j,j) \\
&= \delta_{l,j} -s(l,j) = -s(l,j).
\end{align*}
The last line follows from a well known Stirling number identity. We use induction to conclude $$\E[N_{k-1}^l] \rightarrow \E[N_{k-1}]^l - \sum_{j=1}^{l-1} s(l,j) \E[N_{k-1}^j],$$ thus $$\E[(N_{k-1})_j] = \sum_{j=1}^l s(l,j) \E \left[ N_{k-1}^j \right] \rightarrow \E[N_{k-1}]^l$$ for any fixed $l$. It follows that $N_{k-1}$ converges in distribution to $\Poi ( \mu)$ with $\mu = \E[N_{k-1}]$, completing our proof.

\end{proof}
\end{section}

\section*{Acknowledgements}
We would like to thank Christopher Hoffman, Matthew Junge, and Matthew Kahle for helpful conversations on this subject and readings of early versions of this manuscript.

\bibliographystyle{amsalpha}
\bibliography{references}

\newcommand{\etalchar}[1]{$^{#1}$}
\providecommand{\bysame}{\leavevmode\hbox to3em{\hrulefill}\thinspace}
\providecommand{\MR}{\relax\ifhmode\unskip\space\fi MR }
\providecommand{\MRhref}[2]{%
  \href{http://www.ams.org/mathscinet-getitem?mr=#1}{#2}
}
\providecommand{\href}[2]{#2}
\begin{thebibliography}{AL{\L}M13}

\bibitem[AL15]{tophomvanish}
Lior Aronshtam and Nathan Linial, \emph{When does the top homology of a random
  simplicial complex vanish?}, Random Structures \& Algorithms \textbf{46}
  (2015), no.~1, 26--35.

\bibitem[AL{\L}M13]{tophomcollapse}
Lior Aronshtam, Nathan Linial, Tomasz {\L}uczak, and Roy Meshulam,
  \emph{Collapsibility and vanishing of top homology in random simplicial
  complexes}, Discrete \& Computational Geometry \textbf{49} (2013), no.~2,
  317--334.

\bibitem[BHK11]{fundamental2}
Eric Babson, Christopher Hoffman, and Matthew Kahle, \emph{The fundamental
  group of random 2-complexes}, Journal of the American Mathematical Society
  \textbf{24} (2011), no.~1, 1--28.

\bibitem[B{\'S}97]{garland}
Werner Ballmann and J~{\'S}wi{\c a}tkowski, \emph{On l2-cohomology and property
  (t) for automorphism groups of polyhedral cell complexes}, Geometric \&
  Functional Analysis GAFA \textbf{7} (1997), no.~4, 615--645.

\bibitem[CF14]{douchers}
A~Costa and M~Farber, \emph{Random simplicial complexes}, arXiv preprint
  arXiv:1412.5805 (2014).

\bibitem[ER59]{erdos}
P~Erd\H{o}s and A~R\'enyi, \emph{On random graphs i.}, Publ. Math. Debrecen
  \textbf{6} (1959), 290--297.

\bibitem[Hat02]{hatcher2002algebraic}
Allen Hatcher, \emph{Algebraic topology}, Cambridge University Press, 2002.

\bibitem[HKP12]{spec}
Christopher Hoffman, Matthew Kahle, and Elliot Paquette, \emph{Spectral gaps of
  random graphs and applications to random topology}, arXiv preprint
  arXiv:1201.0425 (2012).

\bibitem[HKP13]{randk2}
\bysame, \emph{The threshold for integer homology in random d-complexes}, arXiv
  preprint arXiv:1308.6232 (2013).

\bibitem[Kah09]{general}
Matthew Kahle, \emph{Topology of random clique complexes}, Discrete Mathematics
  \textbf{309} (2009), no.~6, 1658--1671.

\bibitem[Kah14a]{sharp}
\bysame, \emph{Sharp vanishing thresholds for cohomology of random flag
  complexes}, Ann. of Math. \textbf{179} (2014), 1085--1107.

\bibitem[Kah14b]{survey}
\bysame, \emph{Topology of random simplicial complexes: a survey}, AMS Contemp.
  Math \textbf{620} (2014), 201--222.

\bibitem[KM{\etalchar{+}}13]{clt}
Matthew Kahle, Elizabeth Meckes, et~al., \emph{Limit the theorems for betti
  numbers of random simplicial complexes}, Homology, Homotopy and Applications
  \textbf{15} (2013), no.~1, 343--374.

\bibitem[Koz10]{kozlovthreshold}
Dmitry Kozlov, \emph{The threshold function for vanishing of the top homology
  group of random ??-complexes}, Proceedings of the American Mathematical
  Society \textbf{138} (2010), no.~12, 4517--4527.

\bibitem[LM06]{rand2}
Nathan Linial and Roy Meshulam, \emph{Homological connectivity of random
  2-complexes}, Combinatorica \textbf{26} (2006), no.~4, 475--487.

\bibitem[MW09]{randk}
Roy Meshulam and Nathan Wallach, \emph{Homological connectivity of random
  k-dimensional complexes}, Random Structures \& Algorithms \textbf{34} (2009),
  no.~3, 408--417.

\end{thebibliography}

\end{document}